\documentclass[aos]{imsart}

\RequirePackage{amsthm,amsmath,amsfonts,amssymb,float,mathrsfs,subfigure}
\RequirePackage[numbers,sort&compress]{natbib}
\RequirePackage[colorlinks,citecolor=blue,urlcolor=blue]{hyperref}
\RequirePackage{graphicx}

\startlocaldefs
\theoremstyle{plain}

\newtheorem{theorem}{Theorem}[section]
\newtheorem{lemma}[theorem]{Lemma}
\newtheorem{corollary}[theorem]{Corollary}
\theoremstyle{definition}

\newtheorem{assumption}[theorem]{Assumption}
\newtheorem{remark}[theorem]{Remark}

\endlocaldefs

\newcommand{\bbe}{\mathbb{E}}

\newcommand{\Tr}{\operatorname{Tr}}

\begin{document}

\begin{frontmatter}
\title{Phase Transition of Spectral Fluctuations in Large Gram Matrices with a Variance Profile: A Unified Framework for Sparse CLTs}
\runtitle{Phase Transition of Spectral Fluctuations in Sparse  Gram Matrices}

\begin{aug}
\author[A]{\fnms{Rui}~\snm{Wang} \ead[label=e1]{wangrui\_math@stu.xjtu.edu.cn}},
\author[B]{\fnms{Guangming}~\snm{Pan}\ead[label=e2]{gmpan@ntu.edu.sg}}
\and
\author[A]{\fnms{Dandan}~\snm{Jiang}\ead[label=e3]{jiangdd@mail.xjtu.edu.cn}}\thanks{Corresponding author.}
\address[A]{School of Mathematics and Statistics, Xi’an Jiaotong University \printead[presep={ ,\ }]{e1,e3} }

\address[B]{Division of Mathematical Sciences, Nanyang Technological University \printead[presep={ ,\ }]{e2}}
\end{aug}

\begin{abstract}
We study the asymptotic spectral behavior of high-dimensional random Gram matrices with sparsity and a  variance profile, motivated by applications in wireless communications. Specifically, we consider the Gram matrices $\mathbf S_n=\mathbf Y_n\mathbf Y_n^*$, where the entries of $\mathbf Y_n$ are independent, centered, heteroscedastic, and sparse through Bernoulli masking. The sparsity level is parameterized as $s=q^2/n$, where $q$ ranges from polynomial order up to order $n^{1/2}$. 

We investigate two asymptotic regimes: a moderate-sparsity regime with fixed $s\in(0,1]$, and a high-sparsity regime where $s\to0$. In both regimes, we establish the convergence of the empirical spectral distribution of $\mathbf S_n$ to a deterministic limit, and further derive central limit theorems for linear spectral statistics using resolvent techniques and martingale difference arguments.
Our analysis reveals a phase transition in the fluctuation behavior across the two regimes.  In  the high-sparsity regime, the asymptotic fluctuations are entirely governed by fourth-moment effects, with sparsity-scaled contributions being suppressed. Moreover, the leading deterministic term and the variance of the linear spectral statistic scale at different rates in  $q$, causing the standard centering to fail and necessitating an explicit correction to recover a valid CLT.
The results apply to both Gaussian and non-Gaussian entries and are illustrated through applications to hypothesis testing and outage probability analysis in large-scale MIMO systems.

\end{abstract}



\end{frontmatter}

\section{Introduction}

We consider a $p \times n$ random matrix $\mathbf{Y}_n = (y_{ij})$ with entries 
\begin{equation}\label{eq:model}
	y_{ij} = \frac{1}{\sqrt{ns}}\, b_{ij} w_{ij},
\end{equation}
where $\{b_{ij}\}$ are independent $\rm Bernoulli(1,s)$ random variables, representing random sparsity with retention probability $s = q^2/n$, and $\{w_{ij}\}$ are independent mean-zero random variables with heterogeneous variances $\mathbb{E}|w_{ij}|^2 = \sigma_{ij}^2$. The corresponding Gram matrix is $\mathbf{S} = \mathbf{Y}_n \mathbf{Y}_n^{*}$. This model captures two key features frequently encountered in modern high-dimensional data: structural sparsity and variance heterogeneity.

A primary motivation arises from wireless communications. In multiple-input multiple-output (MIMO) channels, $\mathbf{Y}_n$ may represent the channel matrix, where $b_{ij}$ models random link availability and $\sigma_{ij}^2$ characterizes non-uniform propagation effects. The associated Gram matrix $\mathbf{S}$ determines essential performance metrics, such as the mutual information $C(\sigma^2) = \log \det (1 + \mathbf{S}/\sigma^2)$ with a given parameter $\sigma^2> 0$,
and its fluctuation governs the outage probability—the probability that $C(\sigma^2)$ falls below a target transmission rate. The study of linear spectral statistics (LSS) for $\mathbf{S}$ provides a natural probabilistic framework for quantifying these fluctuations; see, for instance, \cite{Hachem2008,hachem2008-1, bao2015, hu2019, zhang2024, zhang2025}. Consequently, central limit theorems (CLTs) for LSS of Gram-type matrices have become indispensable tools in the stochastic analysis of mutual-information variations.

Beyond capacity analysis, similar ideas facilitate statistical inference on large-scale fading structures across heterogeneous environments. In sparse MIMO systems, the channel gain can be modeled as $y_{ij}=b_{ij}l_{ij}g_{ij}/\sqrt{ns}$, where $\{l_{ij}\}$ represent large-scale fading coefficients, $\{g_{ij}\}$ are  independent and identically distributed (i.i.d.)  small-scale fading. 
 The matrix $\mathbf{L}= (l_{ij})$ varies slowly with the macroscopic environment, and testing whether two scenarios share the same large-scale fading pattern offers an  approach to detecting environmental or topological changes \cite{TMC2018, WCNC2020, TWC2023}. 
In this context, CLTs for LSS of sparse and heteroscedastic Gram matrices serve as a unified framework linking random matrix theory with practical inference tasks in high-dimensional sparse MIMO systems.

The present work stands within a rich tradition of random matrix theory, which provides the foundational framework for analyzing high-dimensional statistical models. Early developments established CLTs for eigenvalue statistics of Wigner ensembles \cite{Sinai1998, BY2005, Lytova2009} and Gaussian sample covariance matrices \cite{jonsson1982, bai2004}, later extended to non-Gaussian data \cite{Pan2008} and, for the specific case of the logarithm function, to variance-profile Gram matrices \cite{Hachem2008,hachem2012clt}. Subsequent work covered Fisher matrices \cite{Zheng2012, Zheng2015} and various high-dimensional models, such as separable covariance models \cite{bai2019,li2024}. More recently, attention has shifted to sparse structures: CLTs and functional limit theorems have been derived for different sparse regimes \cite{shcherbina2010, shcherbina2012, Dumitriu2013, Dumitriu2023, Kanazawa2024, Zhu2024}, with further implications for graph-based learning \cite{Keriven2020}.
Parallel developments have addressed sample covariance matrices under missing data \cite{Jurczak2017,Jong2019, li2024spectral}, albeit under modeling frameworks distinct from ours.
However, existing results primarily concern adjacency-type matrices and sample covariance matrices, and rarely accommodate both sparsity and variance heterogeneity simultaneously. CLTs for sparse and heteroscedastic Gram‑type matrices, despite their importance for the MIMO‑related problems described earlier, have not been systematically studied. This serves as the primary motivation for the present paper.

In this paper, we investigate spectral fluctuations  of large sparse Gram matrices with a given  variance profile.  Our main focus lies in understanding the role of the sparsity parameter $s$ in shaping the asymptotic fluctuations of the LSS. We consider two distinct asymptotic regimes: 
\begin{enumerate}
	\item[(1)] the \emph{moderate-sparsity regime}, where $s$ is a constant in $(0,1]$, and 
	\item[(2)] the \emph{high-sparsity regime}, where $s = q^2/n \to 0$ with $n^{\phi} \leq q < n^{1/2}$ for some fixed $\phi > 0$.
\end{enumerate}
These two regimes jointly characterize the transition from moderately sparse to extremely sparse structures. 
For both regimes, we first establish the convergence of the empirical spectral distribution (ESD) of $\mathbf S_n$ to a nonrandom limiting spectral distribution (LSD). Building on this foundation, we then derive CLTs for the LSS using a martingale difference approach. The mean and variance terms are determined through the unique solutions of certain systems of linear functional equations. When proving the CLT, additional care is required: the analysis necessitates a fourth-order Lindeberg condition, which holds under the existence of moments of order $8+\epsilon$ for some small $\epsilon>0$. 

The contributions of this paper can be summarized as follows. First,  we identify a clear phase transition in the fluctuation behavior between the two sparsity regimes. While the high-sparsity CLT can be viewed as a limiting case of the moderate-sparsity regime as $s \to 0$, the fluctuation structure changes qualitatively: certain components vanish, and the limiting distribution is determined solely by the remaining terms. 
In particular, in the high-sparsity regime, the limiting Gaussian fluctuations are dominantly driven  by fourth-order moment effects, while contributions arising from sparsity-dependent terms vanish asymptotically.
Second, we uncover a mismatch in convergence rates for the mean and variance terms of the centered LSS. After normalization by $\sqrt{p}q$, the variance remains of order one, whereas the corresponding mean term grows at rate $\sqrt{p}/q$, diverging in the high-sparsity regime. While this effect is asymptotically negligible under moderate sparsity, it invalidates the CLT when $q$ is too small.
To address this issue, we introduce a centering correction by subtracting the diverging mean term, thereby obtaining a valid limiting distribution. Third, we develop a unified and general framework applicable to both  Gaussian and non-Gaussian entries, and extending beyond the logarithmic function to a broad class of analytic test functions.

To illustrate the practical relevance of our theoretical results, we present two applications in wireless communication theory.  The first concerns the equality test of two large-scale fading matrices in sparse MIMO systems. The second application involves outage probability analysis in sparse MIMO channels, where the mutual information is expressed through the derived spectral distribution.
These examples demonstrate how the developed random matrix theory connects asymptotic spectral analysis with practical inference and performance evaluation in high-dimensional sparse systems.

The remainder of the paper is organized as follows. Section~2 introduces the model formulation and preliminary results. Section~3 presents the main theorems for LSS of heterogeneous-variance Gram matrices under both sparsity regimes. Technical proofs are deferred to Appendix, with supplementary details provided in the supplementary material. Section~4 discusses the two applications in detail: equality testing of large-scale fading matrices and outage probability analysis in sparse MIMO systems. Section 5 provides numerical simulations to validate the applications discussed in the preceding sections.

\section{The model and preliminaries}

This section introduces the random matrix model and collects several preliminary results that will be used in the proofs of the main theorems. We first specify the model assumptions and notations, followed by key lemmas concerning the  moment bounds.
\subsection{The model}
Let $p=p(n)$ be a sequence of integers such that $\lim _{n \rightarrow \infty} {p}/{n}=c \in(0, \infty)$. 
Consider a $p\times n$ real- or complex-valued random matrix
$\mathbf Y_n=(y_{ij})$ with entries
\begin{equation}\label{model}
	y_{ij}=\frac{1}{\sqrt{ns}}\,b_{ij}w_{ij},
\end{equation}
where the random variables $(b_{ij})$ and $(w_{ij})$ satisfy the following assumptions.

\begin{assumption}\label{ass1}
	Assume that  $(w_{i j}; 1 \leq i \leq p, 1 \leq j \leq n)$  are real- or complex-valued independently  distributed random variables,  satisfying the moment conditions
	$$
	\bbe\left(w_{i j}\right)=0, \quad \mathbb{E}|w_{i j}|^2=\sigma_{ij}^2,~ \text{and}~ \mathbb{E}w_{i j}^2=0 ~ \text{in the complex case},
	$$
	where $\bbe$ denotes expectation. Moreover, there exists some small $\epsilon>0$ such that
	$$
	\sup_{i,j}\bbe \left|w_{i j}\right|^{8+\epsilon}<\infty.
	$$  
\end{assumption}

For a random variable $Z$, write $\operatorname{var}(Z)$ for its variance. Since $\operatorname{var}(y_{ij}) = \sigma_{ij}^{2}/n$, we refer to the collection $(\sigma_{ij}^{2})$ as the variance profile of the model.

\begin{assumption}\label{ass2}
	There exists a finite positive real number $\sigma_{\max }$ such that the family of real numbers ($\sigma_{i j}(n), 1 \leq i \leq p, 1 \leq j \leq n, n \geq 1$) satisfies 
	$$
	\sup _{n \geq 1} \max _{\substack{1 \leq i \leq p \\ 1 \leq j \leq n}}\left|\sigma_{i j}(n)\right| \leq \sigma_{\max }.
	$$
\end{assumption}
\begin{assumption}\label{ass3}
	Furthermore, there exists a real number $\sigma_{\text {min }}^2>0$ such that
	$$
	\liminf _{n \geq 1} \min _{1 \leq j \leq n} \frac{1}{n} \sum_{i=1}^p \sigma_{i j}^2(n) \geq \sigma_{\min }^2.
	$$
\end{assumption}
The lower bound in Assumption~\ref{ass3} ensures mild non-degeneracy of the column-wise average variances. Although a nontrivial LSD can be obtained under weaker conditions \cite{Hachem2007}, this assumption is crucial for stable resolvent estimates and is commonly imposed in CLTs \cite{Hachem2008}. 

\begin{assumption}\label{ass4}
	Assume that $(b_{ij}; 1 \leq i \leq p, 1 \leq j \leq n)$ are  i.i.d. Bernoulli random variables with parameter $s$,  independent of $(w_{ij})$. The sparsity parameter $q$ is introduced through $s=q^2/n$ with $n^{\phi} \leq q \leq n^{1/2}$ for some fixed $\phi > 0$. 
\end{assumption}

Define the matrix $\mathbf X_n$ with the  entries 
$x_{ij}={b_{ij}w_{ij}}/{\sqrt{s}}$. 
Writing $\mathbf Y_n=(\mathbf y_1,\ldots,\mathbf y_n)$ and
$\mathbf X_n=(\mathbf x_1,\ldots,\mathbf x_n)$ with
$\mathbf y_j=\mathbf x_j/{\sqrt n}$, the
Gram matrix can be expressed as
$$
\mathbf S_n=\mathbf Y_n\mathbf Y_n^*=\frac{1}{n}\mathbf X_n\mathbf X_n^*=\frac{1}{ns}(\mathbb B \circ  \mathbf W_n)(\mathbb B \circ \mathbf W_n)^*,
$$
where $\circ$ denotes the Hadamard (element-wise) product,
$\mathbb B=(b_{ij})$, and $\mathbf W_n=(w_{ij})$.
Equivalently,
$$
\mathbf S_n=\sum_{j=1}^n \mathbf y_j\mathbf y_j^*=\frac{1}{n}\sum_{j=1}^n \mathbf x_j\mathbf x_j^*=\frac{1}{ns}\sum_{j=1}^n \mathbb B_j\mathbf w_j\mathbf w_j^*\mathbb B_j,
$$
where $ \mathbb B_j=\operatorname{diag}\left(b_{i j}, 1 \leq i \leq p\right)$ and $\mathbf w_j=(w_{1j},\cdots,w_{pj}) \in \mathbb R^p$.

\subsection{Notations}

Throughout this paper, $\mathbf{i}$ denotes the imaginary unit $\sqrt{-1}$, and for any complex number $z \in \mathbb{C}$, $\Im(z)$ stands for its imaginary part. We write $\mathbb{R}^{+} = \{x \in \mathbb{R}: x \ge 0\}$ and $\mathbb{C}^{+} = \{z \in \mathbb{C}: \Im(z) > 0\}$. The indicator function for any event $\mathcal{A}$ is denoted by $I(\mathcal{A})$. For convergence, $\xrightarrow{\mathcal{D}}$ refers to convergence in distribution and $\xrightarrow{\mathcal{P}}$ to convergence in probability.

For matrices, $\operatorname{diag}(a_i; 1 \le i \le k)$ denotes the $k \times k$ diagonal matrix with entries $a_1, \dots, a_k$. For a matrix $\mathbf{A}$, its $(i,j)$-th entry is denoted by either $a_{ij}$ or $[\mathbf{A}]_{ij}$, depending on context. We write $\mathbf{A}^T$ for the transpose, $\mathbf{A}^*$ for the conjugate transpose, $\operatorname{Tr}(\mathbf{A})$ for the trace, and $\det(\mathbf{A})$ for the determinant when $\mathbf{A}$ is square. The ESD of $\mathbf{A}\mathbf{A}^*$ is defined by
\[
F^{\mathbf{A}\mathbf{A}^*}(x) = \frac{1}{p}\sum_{i=1}^{p} I(\lambda_i(\mathbf{A}\mathbf{A}^*) \le x),
\]
where $\lambda_i(\mathbf{A}\mathbf{A}^*)$ are the eigenvalues of the $p\times p$ matrix $\mathbf{A}\mathbf{A}^*$.

For a complex-valued random variable $Z$ and $\ell > 0$, denote its $\ell$-norm by $\|Z\|_\ell = (\mathbb{E}|Z|^\ell)^{1/\ell}$.
For a vector $\mathbf{v}$, $\|\mathbf{v}\|$ denotes the Euclidean norm and $\|\mathbf{v}\|_{\infty}$ the maximum norm. For a matrix $\mathbf{A}$, $\|\mathbf{A}\|$ denotes the spectral norm, $\|\mathbf{A}\|_{\infty} = \max_{1 \le i \le p} \sum_{j=1}^p |a_{ij}|$ denotes the maximum row-sum norm, and $\rho(\mathbf{A})$ denotes the spectral radius. When no confusion arises, subscripts and superscripts $n$ are omitted for simplicity, and $K$ denotes a generic positive constant whose value may change from line to line.


\subsection{Preliminary results}

We begin by defining the following diagonal matrices, which will be used throughout the paper:
\[ \mathbf \Sigma_{j}  =\operatorname{diag}\left(\sigma^2_{i j}, 1 \leq i \leq p\right),\ \mathbf {\widetilde{\Sigma}}_{i}=\operatorname{diag}\left(\sigma^2_{i j}, 1 \leq j \leq n\right).\]

The following lemma provides a foundational moment identity for quadratic forms, which is crucial for subsequent variance and covariance calculations in our CLT derivation. 
Notably, the sparsity parameter $s$ appears only in the fourth-moment term, reflecting its exclusive role in governing the tail behavior rather than the second-order structure.

\begin{lemma}\label{split}
	Assuming $\mathbf{x}_j=\mathbb B_j\mathbf w_j/\sqrt{s}$ follows our model, we consider any $p \times p$ nonrandom symmetric matrices $\mathbf{A}$ and  $\mathbf{B}$. In this context, we establish the following result:
	\begin{align*}
		&\mathbb{E}[(\mathbf{x}_j^*\mathbf{A}\mathbf{x}_j-\operatorname{Tr}\mathbf{A}\mathbf \Sigma_j)(\mathbf{x}_j^*\mathbf{B}\mathbf{x}_j-\operatorname{Tr}\mathbf{B}\mathbf \Sigma_j)]\\
		=&(\kappa+1)\operatorname{Tr}\mathbf{A} \mathbf \Sigma_j\mathbf{B} \mathbf \Sigma_j+(\frac{\tilde \nu_4}{s}-\kappa-2)\operatorname{Tr}\left[\mathbf A\circ \mathbf B \circ\mathbf \Sigma_j^2\right],
	\end{align*}
	where $\kappa=1$ for the real case and 0 for the complex case, and $\tilde{\nu}_4 = \mathbb{E}\!\left[\left|{w_{ij}}/{\sigma_{ij}}\right|^4\right]$
is the standardized fourth moment of the entries $\{w_{ij}\}$.
\end{lemma}

To control the growth of moments in our martingale analysis, we establish the following bound for centered quadratic forms. The result highlights how the sparsity parameter $q$ moderates the moment growth. 

\begin{lemma}\label{quad_ineq}
	Assume $\mathbf{x}_j$ follows our model, i.e. $\mathbf{x}_j=\mathbb B_j\mathbf w_j/{\sqrt{s}}$. Let $\mathbf{A}=$ $\left(a_{j k}\right)$ be a $p \times p$ nonrandom symmetric matrix and $\sup_{i,j}\bbe\left|w_{ij}\right|^{\ell}/{|\sigma_{ij}|^\ell} \leq \nu^\prime_{\ell}$. Then for $\ell\geq2$,
	$$
    \bbe\left|\mathbf{x}_j^*\mathbf{A}\mathbf{x}_j-\operatorname{Tr}\mathbf{A}\mathbf \Sigma_j\right|^{\ell} \leq C_{\ell}\left[\left( \nu_4^\prime\frac{n^2}{q^2} \|\mathbf A\|^2\right)^{\ell/2}+\nu_{2 \ell}^\prime \frac{n^\ell}{q^{2\ell-2}}\|\mathbf A\|^{\ell}\right],
	$$
	where $C_{\ell}$ is a constant depending on $\ell$ only.
\end{lemma}

The CLT proof via contour integration requires the spectral norm $\|\mathbf{S}_n\|$ to be uniformly bounded. The next lemma verifies this, showing it is of constant order with high probability. Here, an event sequence $\{\mathcal A_n\}_{n\ge1}$ is said to hold with high probability  if, for any $\ell > 0$,
$1- \mathbb{P}(\mathcal A_n) = o(n^{-\ell})$ as $n \to \infty.$ 
\begin{lemma}\label{norm}
Under Assumptions \ref{ass1}--\ref{ass2} and \ref{ass4}, we have, with high probability,
\[
\|\mathbf S_n\|\le \sigma_{\max}^{2}(1+\sqrt c)^{2}.
\]
\end{lemma}

The proofs of the above lemmas are deferred to the supplementary material.

\section{Main results}
This section presents the main theoretical contributions of this work. We first establish the almost sure convergence of the ESD of the sparse Gram matrix with a given variance profile to a deterministic LSD. We then develop CLTs for LSS across both moderate- and high-sparsity regimes.

\subsection{The limiting spectral distribution}
We begin by recalling the definition of the Stieltjes transform. For a probability measure
$v$ on $\mathbb{R}$, its Stieltjes transform is the analytic function  $m_v: \mathbb{C}\setminus\mathbb{R}^+ \rightarrow \mathbb{C}\setminus\mathbb{R}^+$, defined by
$$
m_v(z):=\int_{\mathbb{R}} \frac{\mathrm{d} v(x)}{x-z} \quad\left(z \in \mathbb{C}\setminus\mathbb{R}^+\right).
$$
We shall denote by $ \mathscr S \left(\mathbb{R}^{+}\right)$ the set of Stieltjes transforms of probability measures with support on $\mathbb{R}^{+}$.

The resolvent of a matrix plays a central role in random matrix theory, as it is closely linked to the Stieltjes transform of the ESD.  For any $z \in \mathbb{C} \setminus \mathbb{R}^{+}$, we define
$$\begin{array}{ll}\mathbf Q_n(z)=\left(\mathbf Y_n \mathbf Y_n^*-z \mathbf I_p\right)^{-1}=\left(q_{i j}(z)\right)_{1 \leq i, j \leq p},  \\ \widetilde{\mathbf Q}_n(z)=\left(\mathbf Y_n^* \mathbf Y_n-z \mathbf I_n\right)^{-1}=\left(\tilde{q}_{i j}(z)\right)_{1 \leq i, j \leq n}.\end{array}$$
Let $m_n(z):=\operatorname{Tr} \mathbf Q_n(z)/p$ and $\underline{m}_n(z):= \operatorname{Tr}\widetilde{\mathbf{Q}}_n(z)/n$, which  are exactly the Stieltjes transforms of the ESDs of
$\mathbf Y_n \mathbf Y_n^*$ and $\mathbf Y_n^*\mathbf Y_n$, respectively. 
The following theorem establishes their almost sure convergence  to  deterministic equivalents;  the proof is provided in Appendix \ref{proof-th1}.

\begin{theorem}\label{first-order}
    Let $\{\mathbf{Y}_n \mathbf{Y}_n^*\}$ be a family of random matrices satisfying
    Assumptions~\ref{ass1}--\ref{ass2} and~\ref{ass4}. 
    Moreover, Assumption~\ref{ass1} can be relaxed to require only that
    $w_{ij}$ have  uniformly bounded $(4+\epsilon)$-th moments for some $\epsilon>0$. Define the system of $p + n$ equations:
    \begin{equation}
        \begin{cases}
            t_{i}(z) = \dfrac{-1}{z[1 + (1 / n) \operatorname{Tr} ( \widetilde{\mathbf{\Sigma}}_{i} \widetilde{\mathbf{T}}_n(z) ) ]}, & \text{for } 1 \leq i \leq p, \\
            \tilde{t}_{j}(z) = \dfrac{-1}{z\left[1 + (1 / n) \operatorname{Tr} \left( \mathbf{\Sigma}_{j} \mathbf{T}_n(z) \right) \right]}, & \text{for } 1 \leq j \leq n,
        \end{cases} \label{equation}
    \end{equation}
    where
    \begin{align*}
        \mathbf{T}_n(z) = \operatorname{diag}\left(t_{i}(z),\; 1 \leq i \leq p\right), \quad
        \widetilde{\mathbf{T}}_n(z) = \operatorname{diag}\left(\tilde{t}_{j}(z),\; 1 \leq j \leq n\right).
    \end{align*}
    Then the following  statements hold:
    \begin{enumerate}
        \item The system admits a unique solution $\left(t_{1}, \ldots, t_{p}, \tilde{t}_{1}, \ldots, \tilde{t}_{n}\right) \in \mathscr{S}(\mathbb{R}^{+})^{p+n}$. In particular, $m_n^0(z) =  \operatorname{Tr} \mathbf{T}_n(z)/p$ and $ \underline{m}_n^0(z) =  \operatorname{Tr} \widetilde{\mathbf{T}}_n(z)/n$ belong to $\mathscr S(\mathbb R^+)$.
        There exist probability measures $\pi_n$ and $\tilde\pi_n$ on $\mathbb R^+$ such that
        \begin{align*}
            m_n^0(z) = \int_0^\infty \frac{\pi_n(\mathrm{d}\lambda)}{\lambda - z}, \quad
            \underline{m}_n^0(z) = \int_0^\infty \frac{\tilde{\pi}_n(\mathrm{d}\lambda)}{\lambda - z}.
        \end{align*}

        \item Almost surely, for all $z \in \mathbb{C} \setminus \mathbb{R}^{+}$,
        \[
        m_n(z) - m_n^0(z) \to 0, \quad 
        \underline{m}_n(z) - \underline{m}_n^0(z) \to 0 \quad \text{as } n \to \infty.
        \]
    \end{enumerate}
\end{theorem}

\begin{remark}
    Under Assumption~\ref{ass2}, \cite{Hachem2007} derived a deterministic equivalent for the Stieltjes transform of $\mathbf{Y}_n\mathbf{Y}_n^*$ in the non-sparse model $y_{ij} = n^{-1/2} x_{ij}$, where $\mathbb{E} x_{ij} = 0$, $\mathbb{E}|x_{ij}|^2 = \sigma_{ij}^2$, and the $x_{ij}$ have uniformly bounded $(4+\epsilon)$-th moments. Theorem~\ref{first-order} recovers this result when $s = 1$, $b_{ij} \equiv 1$, and $w_{ij} = x_{ij}$ satisfy the same moment conditions. 
    More importantly, Theorem~\ref{first-order} holds under  weaker assumptions, allowing the $(4+\epsilon)$-th moments of $x_{ij} = b_{ij} w_{ij}/\sqrt{s}$ to be unbounded as $s \to 0$, thus capturing the tail amplification induced by sparsity. By accommodating general sparsity levels $s \in (0,1]$, it substantially extends the scope of deterministic equivalents beyond \cite{Hachem2007}.
\end{remark}


\subsection{CLT in the moderate-sparsity regime}

We consider the LSS of the Gram matrix $\mathbf S_n=\mathbf Y_n \mathbf Y_n^*$, defined for a test function $f \in \mathcal{A}$ as
$$
L_n(f, \mathbf S_n)=\int f(x) \mathrm{d} F^{\mathbf S_n}(x),
$$
where $\mathcal{A}$ represents a set of analytic functions defined on an open set of the complex plane containing the supporting set of the LSD of  $\mathbf S_n$.
Such statistics naturally arise in various inferential problems, particularly in hypothesis testing regarding the variance structure of the population.

This section focuses on the moderate-sparsity regime, where the retention probability $s \in (0,1]$ remains constant as $n \to \infty$.  
To capture non-degenerate second-order fluctuations, we consider the centered and properly normalized version of the LSS,
\begin{align}\label{Lc}
	L^{c}(f, \mathbf{S}_n)
	= \sqrt{p}\, q
	\int f(x)\, \mathrm{d}\!\left[F^{\mathbf{S}_n}(x) - \pi_n(x)\right],
\end{align}
where $\pi_n$ denotes the deterministic equivalent of the ESD $F^{\mathbf{S}_n}$.  
The scaling factor $\sqrt{p}q$ reflects the variance structure induced by sparsity. While the classical $n$-scaling remains applicable when $s$ is fixed (since $\sqrt{p} q$ is of the same order as $n$), we therefore employ $\sqrt{p}q$ for a unified treatment across both moderate- and high-sparsity regimes. This choice is necessary because, as shown in equation (1) of the supplementary material, it is the only scaling that yields  a non-degenerate limit in the high-sparsity case.


For $z_1, z_2 \in \mathbb{C}\setminus\mathbb{R}^+$, define the $n \times n$ matrix $\mathbf A_n(z_1,z_2)$ with  entries
$$
a_{lm}(z_1,z_2)=\frac{1}{n} \frac{(1 / n) \operatorname{Tr} \mathbf\Sigma_{l} \mathbf T(z_1)\mathbf\Sigma_m \mathbf T(z_2)}{\left[1+(1 / n) \operatorname{Tr} \mathbf\Sigma_{l} \mathbf T(z_1)\right]\left[1+(1 / n) \operatorname{Tr} \mathbf\Sigma_{l} \mathbf T(z_2)\right]}, \quad 1 \leq l, m \leq n.
$$
The main result is stated in the following theorem.
\begin{theorem}\label{th1}
    Let $f_1, \dots, f_k$ be functions that are analytic in a complex neighborhood of the support of the LSD of $\mathbf{S}_n$, and continuous at zero. Suppose Assumptions~\ref{ass1}--\ref{ass4} hold  and $s\in (0,1]$ is fixed. Then the random vector $\left(L^c(f_j, \mathbf{S}_n)\right)_{j=1}^k$
    is tight and converges weakly to a Gaussian vector \(\left(X_{f_j}\right)_{j=1}^k\). The mean and covariance structure of the limit are given as follows.
    \begin{enumerate}
        \item The asymptotic mean is characterized by the $n$-dependent contour integral
        \[
        \mu_n(X_f) = -\frac{1}{2\pi \mathbf i} \oint_{\mathcal{C}_1} f(z) \mathscr{E}_n(z)  \mathrm{d}z,
        \]
        where \(\mathscr{E}_n(z) = \frac{1}{n} \sum_{j=1}^n \psi_{j}(z)\), and \(\{\psi_{j}(z)\}_{j=1}^n\) is the unique solution to the system
        \[
        \psi_j(z) = \sum_{m=1}^n a_{jm}(z, z) \psi_m(z) + \theta_j(z), \quad 1 \le j \le n.
        \]
        Here, \(\theta_j(z)\) is defined as:
        \begin{align*}
            \theta_j(z) = {} & \frac{n\tilde{\nu}_4}{q\sqrt{p}} \left\{ \frac{1}{n^2} z^3 \tilde{t}_j^2(z) \sum_{i=1}^p \sigma_{ij}^2 t_i^3(z) \operatorname{Tr} \left( \widetilde{\mathbf{T}}(z) \circ \widetilde{\mathbf{T}}(z) \circ \widetilde{\mathbf{\Sigma}}_i^2 \right) + \frac{1}{n} z^2 \tilde{t}_j^3(z) \operatorname{Tr} \left( \mathbf{T}(z) \circ \mathbf{T}(z) \circ \mathbf{\Sigma}_j^2 \right) \right\} \\
            & + \frac{q}{\sqrt{p}} \frac{1}{n} z^3 \tilde{t}_j^2(z) \sum_{i=1}^p \sigma_{ij}^2 t_i^3(z) \left\{ \kappa \tilde{U}_i(z) - \frac{\kappa+2}{n} \operatorname{Tr} \left( \widetilde{\mathbf{T}}(z) \circ \widetilde{\mathbf{T}}(z) \circ \widetilde{\mathbf{\Sigma}}_i^2 \right) \right\} \\
            & + \frac{q}{\sqrt{p}} z^2 \tilde{t}_j^3(z) \left\{ \kappa U_j(z) - \frac{\kappa+2}{n} \operatorname{Tr} \left( \mathbf{T}(z) \circ \mathbf{T}(z) \circ \mathbf{\Sigma}_j^2 \right) \right\},
        \end{align*}
        with \((\tilde{U}_i(z))_{1\le i\le p}\) and \((U_j(z))_{1\le j\le n}\) determined by the systems in Lemma \ref{mean-eq}. 

        \item The covariance function for test functions \(f, g\) can be calculated by 
        \[
        \nu_n(X_f, X_g) = -\frac{1}{4\pi^2} \oint_{\mathcal{C}_1} \oint_{\mathcal{C}_2} f(z_1) g(z_2) \mathscr{V}_n(z_1, z_2)  \mathrm{d}z_1 \mathrm{d}z_2,
        \]
        where
        \begin{align*}
            \mathscr V_n(z_1,z_2)=\frac{\partial^2}{\partial z_2 \partial z_1}&\left\{\frac{s(\kappa+1)}{p}\sum_{j=1}^nH_{jj}(z_1,z_2) \right.\\
		& +\left. \frac{\tilde\nu_4-\kappa s-2s}{pn}\sum_{j=1}^n \tilde t_j\left(z_1\right) \tilde t_j\left(z_2\right)\operatorname{Tr}\left[\mathbf T^{(1)}(z_1)\circ \mathbf T^{(1)}(z_2) \circ \mathbf  \Sigma_j^2\right]\right\}. 
        \end{align*}
        Here, \(\mathbf{T}^{(1)}(z) = \left( \mathbf{I} + \frac{1}{n} \sum_{k=1}^n \tilde{t}_k(z) \mathbf{\Sigma}_k \right)^{-1}\), and for each \(j\), \(\{H_{\ell j}(z_1,z_2)\}_{\ell=1}^j\) is the unique solution to the triangular system
        \[
        H_{\ell j}(z_1,z_2) = \sum_{i=1}^{j-1} a_{\ell i}(z_1, z_2) H_{i j}(z_1,z_2) + n a_{\ell j}(z_1,z_2), \quad 1 \le \ell \le j.
        \]
    \end{enumerate}
    
    The contours \(\mathcal{C}_1, \mathcal{C}_2\)  are closed, non-overlapping (for the covariance), positively oriented, and each encloses the support of \(\pi_n\).
\end{theorem}

\begin{remark}
	To verify the consistency of our CLT with those established for complete data in earlier studies, consider the special case where the variance profile matrix $\mathbf{\Sigma} := (\sigma_{ij}^2)$ reduces to $\mathbf{\Sigma}=\sigma^2\mathbf{1}_{p \times p}$ and $s=1$. Then, the model in \eqref{model} becomes $\mathbf y_{j}=\mathbf \Gamma \mathbf z_j/\sqrt{n}$, where $\mathbf \Gamma=\sigma^2\mathbf I_p$, and $(z_{ij})$ are i.i.d. random variables with mean zero and unit variance. The LSD of $\mathbf{\Gamma}$ is the point mass $H(t) = \delta_{\sigma^2}(t)$. In this case, $\tilde t_j(z)$ is identical for all $j$ and is denoted by $\underline m(z)$, which satisfies
	$$
	\underline m(z)=\frac{1}{c\int\frac{t\mathrm d H(t)}{1+t\underline m(z)}-z}.
	$$
	A direct computation shows that, in this setting, our CLT reduces to the classical form for complete data. For more details, see Theorem 1.1 in \cite{bai2004} for the case $\tilde\nu_4 = 3$ and Theorem 1.4 in \cite{Pan2008} for the more general case. The detailed verification can be found in the supplementary material.
\end{remark}


\subsection{CLT in the high-sparsity regime}
We now turn to the CLT for the LSS in the high-sparsity regime, where $s \to 0$ with $n^{\phi} \le q < n^{1/2}$ for some fixed $\phi > 0$.  
Unlike the dense or moderate-sparsity settings, the asymptotic mean term in Theorem~\ref{th1} diverges as $s \to 0$.  
This divergence stems from a mismatch between the fluctuation and centering terms under the CLT normalization: after scaling by $\sqrt{p}q$, the variance of the LSS in \eqref{Lc} remains of order one, whereas the corresponding mean contribution grows at rate $\sqrt{p}/q$.  
Consequently, when $q^2 \ll \sqrt{p}\, q$, the standard centering fails to remove the dominant deterministic component, leading to a degenerate limit.

To restore a valid centering and obtain a finite limit, we refine the definition of the centralized LSS by explicitly subtracting the diverging mean term.  
The corrected statistic is defined as
$$
L^{\tilde c}\left(f,\mathbf S_n\right)=\sqrt{p}q\int f(x) \mathrm d\left[F^{\mathbf{S}_n}(x)-\pi_n(x)\right]+\frac{1}{2 \pi \mathbf i}\oint_{\mathcal C_1} f(z) \mathscr E_n(z)\mathrm d z,
$$
where $\mathscr{E}_n(z) = n^{-1}\sum_{j=1}^n \psi_{j}(z)$, and \(\{\psi_{j}(z)\}_{j=1}^n\) is the unique solution to the system of linear equations
\[
\psi_j(z) = \sum_{m=1}^n a_{jm}(z, z) \psi_m(z) + \theta_j(z), 
\qquad 1 \le j \le n.
\]
Here, $\theta_j(z)$ is given by
\begin{align*}
	\theta_j(z)
	=&\frac{n\tilde \nu_4}{q\sqrt{p}}\left\{\frac{1}{n^2}z^3\tilde t_j^2(z) \sum_{i=1}^p\sigma_{ij}^2t_i^3(z) \operatorname{Tr} \left(\widetilde{\mathbf T}(z)\circ \widetilde{\mathbf T}(z)\circ \widetilde { \mathbf \Sigma}_i^2\right)+\frac{1}{n}z^2\tilde t_j^3\operatorname{Tr} \left(\mathbf T(z)\circ \mathbf T(z)\circ \mathbf  \Sigma_j^2\right)\right\}.
\end{align*}
This additional contour integral compensates for the diverging deterministic bias and yields a properly centered statistic with non-degenerate Gaussian fluctuations.  
The corresponding CLT for the LSS in the high-sparsity regime is summarized in the following corollary.
\begin{corollary}\label{coro1}
    Let $f_1, \dots, f_k$ be functions that are analytic in a complex neighborhood of the support of the LSD of $\mathbf{S}_n$, and continuous at zero.  Under Assumptions \ref{ass1}--\ref{ass4} with $1/4< \phi < 1/2$, the random vector $\left(L^{\tilde c}(f_j, \mathbf{S}_n)\right)_{j=1}^k$
    is tight and converges weakly to a Gaussian vector \(\left(X_{f_j}\right)_{j=1}^k\), where the mean $\mu_n(X_f)=0$ and  the covariance function  between
    $X_f$ and $X_g$ is defined as follows:
	$$
	\begin{aligned}
		\nu_n\left(X_f, X_g\right) =-\frac{1}{4 \pi^2} \oint_{\mathcal C_1}\oint_{\mathcal C_2} f\left(z_1\right) g\left(z_2\right)\mathscr V_n(z_1,z_2)\mathrm d z_1 \mathrm d z_2
	\end{aligned}
	$$
	with 
	$$
	\begin{aligned}
		\mathscr V_n(z_1,z_2)=&\frac{\partial^2}{\partial z_2 \partial z_1}\left\{ \frac{\tilde\nu_4}{pn}\sum_{j=1}^n \tilde t_j\left(z_1\right) \tilde t_j\left(z_2\right)\operatorname{Tr}\left[\mathbf T^{(1)}(z_1)\circ \mathbf T^{(1)}(z_2) \circ \mathbf  \Sigma_j^2\right]\right\}.  
	\end{aligned}
	$$
	The contours \(\mathcal{C}_1, \mathcal{C}_2\)  are closed, non-overlapping, positively oriented, and each encloses the support of \(\pi_n\).
\end{corollary}

\begin{remark}
    Our analysis in the high-sparsity regime focuses on the range $1/4 < \phi < 1/2$. Within this range, the key simplification occurs as the $O(\sqrt{n}/q^{2})$ terms in the mean expansion become negligible, leading to the clean asymptotic form presented herein. For lower sparsity levels ($\phi \le 1/4$), the mean structure involves additional terms that substantially complicate the analysis; the current restriction thus serves to delineate a regime where a clear and concise result can be established.

    To further illustrate this point, we conduct a simulation at the boundary case $q \asymp n^{\phi}$ with $\phi=1/4$. The model follows the sparse MIMO channel  formulated in \eqref{eq:channel_model} of Section~4.2,  with the dimensions $(N_r, N_t)$ therein replaced by $(p, n)$. 
    The diagonal entries of deterministic matrices $\mathbf D$ and $\widetilde {\mathbf D}$ are generated as $d_i \sim \text{Uniform}(6,8)$ for $i=1,\ldots,p$ and $\tilde d_j \sim \text{Uniform}(6,8)$ for $j=1,\ldots,n$. The matrix $\mathbf X$ has i.i.d.\ entries $x_{ij} \sim \mathcal{CN}(0,1)$,  and $\sigma^2$ is set to 1. Throughout, we take $(p, n) = (5000, 10000)$ and average over $2,000$ independent channel realizations. The statistic considered is
    $\widetilde T_{\log}={q}/{\sqrt{n}}\bigl\{ C_{\mathbf H \mathbf H^{*}}(\sigma^{2}) - V(\sigma^{2}) \bigr\}-\mu_{\log}$, where $ C_{\mathbf H \mathbf H^{*}}(\sigma^{2})$ and $\mu_{\log}$ are defined in Theorem \ref{MI}.
    It is centered so that $\widetilde T_{\log}$ converges to $\mathcal N(0,\sigma_{\log}^{2})$.  For fixed $n$, reducing $q$ enhances the magnitude of the mean deviation. This trend is clearly visible when comparing the cases $q = 0.7 n^{1/4}$, $q = 0.6 n^{1/4}$, and  $q = 0.5 n^{1/4}$, where the smaller values of $q$ lead to more pronounced mean drift (Figure~\ref{fig:devation}). This behavior indicates that the $O(\sqrt{n}/q^{2})$ term remains non-negligible when $q \asymp n^{1/4}$. Its contribution is sufficiently large to distort the mean, which in turn motivates the restriction $1/4 < \phi < 1/2$, under which a clean asymptotic expression can be established.
    \begin{figure}[tp]
	\centering
    \subfigure[$q=0.7n^{1/4}$]{\includegraphics[width=0.32\linewidth]{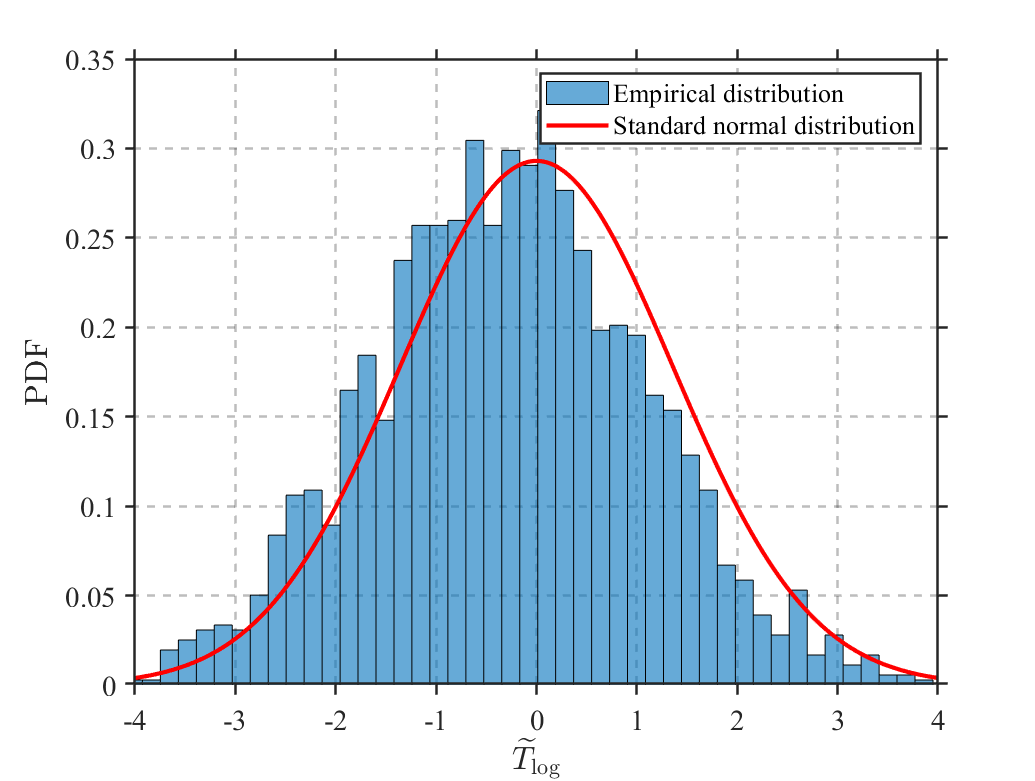}}
    \subfigure[ $q=0.6n^{1/4}$]{\includegraphics[width=0.32\linewidth]{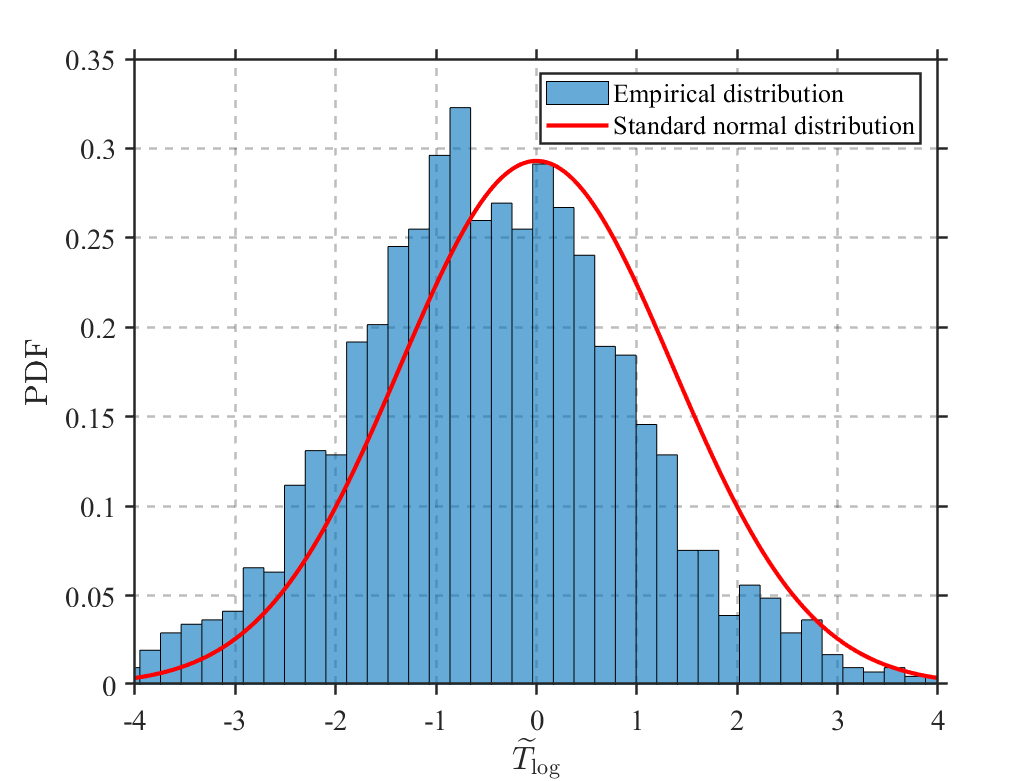}}
    \subfigure[ $q=0.5n^{1/4}$]{\includegraphics[width=0.32\linewidth]{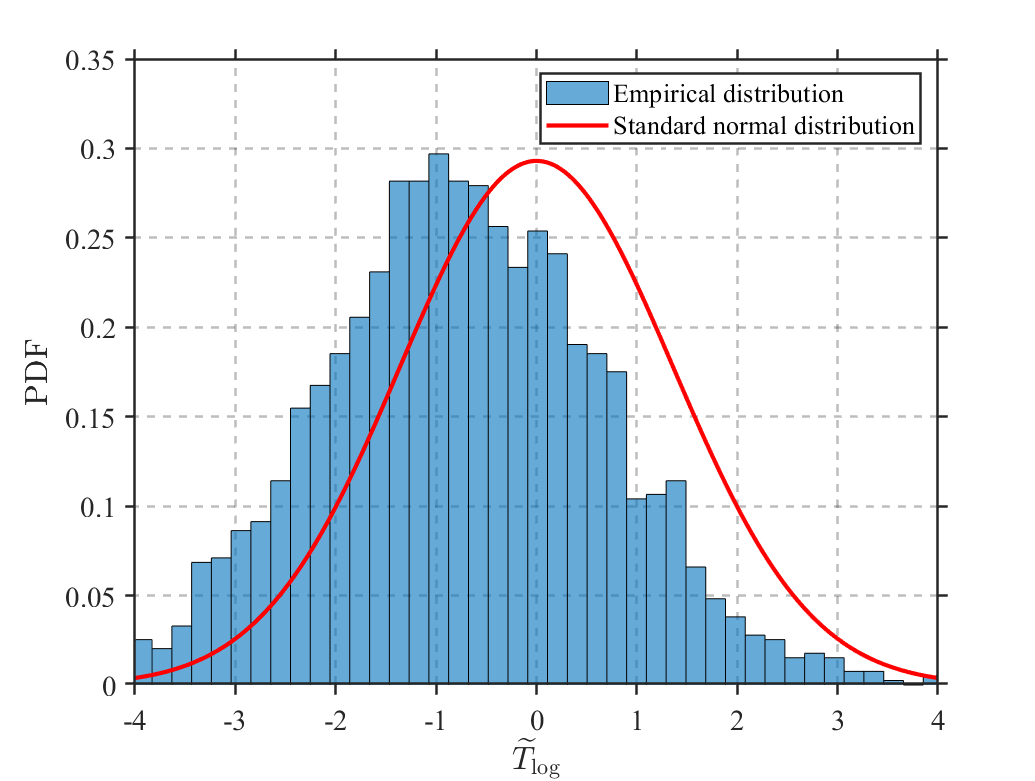}}
	\caption{Comparison of empirical and theoretical distributions of  $\tilde T_{\log}$ when $n=10000$.}
	\label{fig:devation}
\end{figure}

\end{remark}

	The combination of Theorem~\ref{th1} and Corollary~\ref{coro1} highlights two distinct phenomena concerning the asymptotic fluctuations of the LSS.  \emph{(1) Phase transition.}
	As the sparsity parameter $s$ decreases to zero, the limiting law in Corollary~\ref{coro1} emerges as a continuous extension of Theorem~\ref{th1}, exhibiting a clear phase transition between moderate- and high-sparsity.  
	While the LSD remains unchanged, the fluctuation structure shifts from the  moderate-sparsity regime to the high-sparsity regime, illustrating how sparsity affects the second-order limiting behavior.  \emph{(2) Fourth-moment-driven fluctuation.}
    In the high-sparsity regime, the asymptotic mean and variance are governed by the term originating from the fourth-moment structure of the entries, which appears with coefficient $\tilde\nu_4$ after taking the sparsity limit. While the resulting expressions still involve the diagonal second-moment
   structure through $\widetilde{ \mathbf\Sigma}_i^2$ and $\mathbf\Sigma_j^2$, the   components that are originally scaled by $s$ are suppressed. By contrast, under moderate sparsity, multiple fluctuation contributions coexist and jointly determine the asymptotic mean and variance. This highlights a distinct fluctuation regime induced by high sparsity.

Taken together, Theorem~\ref{th1} and Corollary~\ref{coro1} 
establish a unified CLT framework for LSS of large sparse Gram 
matrices. This framework encompasses a broad range of sparsity 
regimes, from moderate-sparsity to high-sparsity, and applies to both Gaussian and non-Gaussian distributions, as well as to  real and complex-valued entries. It thereby provides a comprehensive and robust foundation for analyzing second-order fluctuations in high-dimensional models.

\section{Applications}

\subsection{Application to equality test of two large-scale fading matrices}\label{appl1}

Consider a sparse MIMO communication system with $N_t$ transmit and $N_r$ receive antennas. The channel matrix $\mathbf{H} \in \mathbb{C}^{N_r \times N_t}$ is modeled as 
\begin{equation}
	\mathbf{H} = \frac{1}{\sqrt{N_t s}} (\mathbb{B} \circ \mathbf L)\circ \mathbf{G}, \label{mimo}
\end{equation}
where $\mathbb B=(b_{ij})$ has independent entries
$b_{ij}\sim{\rm Bernoulli}(1,s)$ with sparsity level
$s=q^2/N_t$. The matrix $\mathbf L=(l_{ij})$ contains the  large-scale fading coefficients, whose  slow variation reflects the macroscopic propagation environment.  The matrix $\mathbf G=(g_{ij})$ models i.i.d.\ small-scale fading, typically circular Gaussian $g_{ij}\sim \mathcal{CN}(0,1)$, though more general distributions with finite moments can also be accommodated.

Testing whether two scenarios share the same large-scale fading pattern provides a statistical means to detect environmental or topological changes, thereby supporting network reconfiguration and physical-layer authentication \cite{TMC2018, WCNC2020, TWC2023}.
Suppose there are two independent  high-dimensional channel matrices:
$$
\mathbf H_1 \in \mathbb{C}^{N_r \times N_t}, \quad \mathbf H_2 \in \mathbb{C}^{N_r \times N_t},
$$
whose entries are generated as
$$
h_{ij,1}=\frac{1}{\sqrt{N_t s}}b_{ij,1}l_{ij,1}g_{ij,1}, \quad h_{ij,2}=\frac{1}{\sqrt{N_t s}}b_{ij,2}l_{ij,2}g_{ij,2},
$$
where $b_{ij,1}$ and $b_{ij,2}$ are independent missing indicators, following $\text{Bernoulli}(1,s)$.

The objective is to test whether the large-scale fading matrices $\mathbf L_1$ and $\mathbf L_2$ are the same in two different scenarios, i.e.
$$
\mathcal H_0: \mathbf L_1=\mathbf L_2 \quad \text { vs } \quad \mathcal H_1: \mathbf L_1 \neq \mathbf L_2. 
$$
Define the Gram matrices  as
$$
\mathbf S_1=\mathbf H_1 \mathbf H_1^*, \quad \mathbf S_2= \mathbf H_2\mathbf H_2^*.
$$
Since under model~\eqref{mimo} one has
\[
\mathbb E\!\left[\operatorname{Tr}(\mathbf S_k)\right]
=\frac{1}{N_t}\operatorname{Tr}(\mathbf L_k\mathbf L_k^*), \qquad k=1,2,
\]
the trace difference $\operatorname{Tr}(\mathbf S_1)-\operatorname{Tr}(\mathbf S_2)$ 
provides an unbiased estimate of the large-scale energy gap 
$1/N_t\left[\operatorname{Tr}(\mathbf L_1\mathbf L_1^*)-\operatorname{Tr}(\mathbf L_2\mathbf L_2^*)\right]$.
Motivated by this observation, a natural choice is the linear test
function $f(x)=x$, for which the corresponding LSS reduces to the trace. Accordingly, we define the test statistic
\[
D_x=\frac{q}{\sqrt{N_r}}\left[\operatorname{Tr}(\mathbf S_1)-\operatorname{Tr}(\mathbf S_2)\right].
\]
The following theorem establishes the asymptotic null distribution of $D_x$.
\begin{theorem}\label{th4-1}
	Under $\mathcal H_0$, the standardized test statistic
	$$
	T_x=\frac{D_x-\mu_{H_0}}{\sqrt{2}\sigma_{H_0}} \xrightarrow{\mathcal D}\mathcal N(0,1),
	$$
	where $\mu_{H_0}=0$, and $$\sigma^2_{H_0}=\frac{1}{N_rN_t}\sum_{i,j}\bbe (l_{ij,1} g_{ij,1})^4-\frac{s}{N_rN_t}\sum_{i,j}l_{ij,1}^4.$$
\end{theorem}

\begin{remark}
 The  quantity $1/(N_rN_t)\sum_{i,j}\bbe (l_{ij,1} g_{ij,1})^4$ can be consistently estimated by  $sN_t/N_r\sum_{i,j}h_{ij,1}^4$, which is unbiased (using $\hat s=1/(N_rN_t) \sum_{i,j} I (h_{ij,1}\neq 0)$ if $s$ is unknown). Assuming that 
\(\mathbb{E}(g_{ij,1}^4) = \tilde{\nu}_4\), the second term, ${s}/({N_rN_t})\sum_{i,j}l_{ij,1}^4$,  can be estimated analogously by $sN_t/(N_r\tilde{\nu}_4)\sum_{i,j}h_{ij,1}^4$.

Moreover, the technical restriction $\phi>1/4$ in Theorem~\ref{th1}, 
imposed to control an additional term of order $O(\sqrt{N_t}/q^2)$ in the mean expansion, 
can be removed under $\mathcal{H}_0$. 
In this case, since the mean is identically zero, the potentially divergent term no longer arises, and the CLT in Theorem~\ref{th4-1} remains valid for all 
$N_t^{\phi}\le q \le N_t^{1/2}$ with $\phi>0$.
\end{remark}


Theorem~\ref{th4-1} characterizes the null distribution of the test statistic,  thereby ensuring valid control of the type I error. Given the significance level $\alpha$, the null hypothesis is rejected when
$$
\begin{aligned}
	\left\{D_x>\mu_{H_0}+\sqrt{2}\sigma_{H_0}z_{1-\alpha}\right\}, 
\end{aligned}
$$
where $z_{1-\alpha}$ denotes the $1-\alpha$ quantile of the standard normal distribution.
To evaluate the power of the test, it remains to determine the distribution of $D_x$ under the alternative hypothesis $\mathcal{H}_1:\mathbf{L}_1 \neq \mathbf{L}_2$. In this setting, both the centering and the variance of $D_x$ shift by quantities that depend on the deviation between $\mathbf{L}_1$ and $\mathbf{L}_2$. The next theorem establishes the corresponding CLT under $\mathcal{H}_1$, from which the asymptotic power function follows.
\begin{theorem}\label{power}
	Under $\mathcal H_1$, we have
	\begin{align*}
    \frac{{q}/{\sqrt{N_r}}\left[\operatorname{Tr}(\mathbf S_1)-\operatorname{Tr}(\mathbf S_2)\right]-q/(\sqrt{N_r}N_t)\sum_{i,j}(l_{ij,1}^2-l_{ij,2}^2)-\mu_{H_1}}{\sigma_{H_1}} \xrightarrow{\mathcal D} \mathcal N(0,1),
	\end{align*}
	where $\mu_{H_1}=0$ and $$\sigma_{H_1}^2=\frac{1}{N_rN_t}\sum_{i,j}\left[\bbe (l_{ij,1}g_{ij,1})^4+\bbe (l_{ij,2}g_{ij,2})^4\right]-\frac{s}{N_rN_t}\sum_{i,j}(l_{ij,1}^4+l_{ij,2}^4).$$
	The asymptotic power function is therefore given by
	\begin{align*}
		Power=1-\Phi\left(\frac{ \sqrt{2}z_{1-\alpha}\sigma_{H_0}+\mu_{H_0}-q/(\sqrt{N_r}N_t)\sum_{i,j}(l_{ij,1}^2-l_{ij,2}^2)-\mu_{H_1}}{\sigma_{H_1}} \right).
	\end{align*}
\end{theorem}
\begin{remark}
    Two observations apply under the alternative hypothesis.  
	First, the terms $1/(N_rN_t)\sum_{i,j} \bbe (l_{ij,k}g_{ij,k})^4$ and $s/(N_rN_t)\sum_{i,j} l_{ij,k}^2, k=1,2$ appearing in Theorem \ref{power} can be consistently estimated in the same manner as in Theorem~\ref{th4-1}. 
	Second, under $\mathcal{H}_1$, the restriction on $\phi$ can be relaxed as well,
	since for the  test function $f(x)=x$, the CLT remains valid for all $N_t^{\phi} \le q \le N_t^{1/2}$ with $\phi>0$.
\end{remark}

\subsection{Application to outage probability analysis in  sparse MIMO channels}\label{appl2}


Consider a point-to-point sparse MIMO system with $N_t$ antennas at the transmitter and $N_r$ antennas at the receiver. 
The received signal $\mathbf{y} \in \mathbb{C}^{N_r}$ can be given by
\begin{equation}
	\mathbf{y} = \mathbf{H}\mathbf{s} + \mathbf{n},
	\label{eq:system_model}
\end{equation}
where $\mathbf{s} \in \mathbb{C}^{N_t}$ represents the transmitted signal, $\mathbf{H}$ denotes the $N_r \times N_t$ sparse channel matrix, 
and $\mathbf{n}$ is the additive white Gaussian noise, whose entries are i.i.d. circular Gaussian random variables 
with variance $\sigma^{2}$, i.e., $\mathbb{E}(\mathbf{n}\mathbf{n}^{*}) = \sigma^{2}\mathbf{I}$.
The channel matrix $\mathbf{H}$ is given by
\begin{equation}
	\mathbf{H} 
	=\frac{1}{\sqrt{N_t s}}\mathbb B \circ (\mathbf{D}^{1/2}\mathbf{X}\widetilde{\mathbf{D}}^{1/2}),
	\label{eq:channel_model}
\end{equation}
where $\mathbb B=(b_{ij})$ has independent entries
$b_{ij}\sim{\rm Bernoulli}(1,s)$ with sparsity level
$s=q^2/N_t$. Moreover, $\mathbf{D}$ and $\widetilde{\mathbf{D}}$ are the receive and transmit correlation matrices, which are deterministic diagonal matrices and defined respectively as
\[
\mathbf{D} = \mathrm{diag}(d_1, d_2, \ldots, d_{N_r}), \quad
\widetilde{\mathbf{D}} = \mathrm{diag}(\tilde{d}_1, \tilde{d}_2, \ldots, \tilde{d}_{N_t}).
\]
The random matrix $\mathbf{X}$ consists of i.i.d. complex Gaussian entries with zero mean and unit variance. Then, $\tilde \nu_4=\bbe |x_{ij}|^4=2$.

Under the assumption $\mathbb{E}(\mathbf{s}\mathbf{s}^{*})= \mathbf{I}$, 
the mutual information of the considered MIMO system is given by
\begin{equation}
	C_{\mathbf{H}\mathbf{H}^{*}}(\sigma^{2}) 
	= \log \det \left( 
	\mathbf{I} + \frac{\mathbf{H}\mathbf{H}^{*}}{\sigma^{2}}
	\right),
	\label{eq:MI_expression}
\end{equation}
which is a random variable due to the randomness of the channel matrix $\mathbf{H}$. 
This randomness motivates us to investigate  the fluctuation of 
$C_{\mathbf{H}\mathbf{H}^{*}}(\sigma^{2})$, which can be used to compute  outage probability.

To develop the asymptotic theory of the mutual information, we first introduce some intermediate notations.
Given $z \in \mathbb{C} \setminus \mathbb{R}^{+}$, for $(\delta(z), \tilde{\delta}(z))$  and matrices 
$\mathbf{T}(z), \widetilde{\mathbf{T}}(z)$
satisfying the following system of equations according to Theorem \ref{first-order}:
\begin{equation*}
\left\{
\begin{aligned}
	\delta(z) &= \frac{1}{N_t}\operatorname{Tr}\!\left(\mathbf{D}\mathbf{T}(z)\right), \\
	\tilde{\delta}(z) &= \frac{1}{N_t}\operatorname{Tr}\!\left(\widetilde{\mathbf{D}}\widetilde{\mathbf{T}}(z)\right), \\
	\mathbf{T}(z) &= 
	\Big( -z\!\left(\mathbf{I}+\tilde{\delta}(z)\mathbf{D}\right), \\
	\widetilde{\mathbf{T}}(z) &= 
	\Big( -z\!\left(\mathbf{I}+\delta(z)\widetilde{\mathbf{D}}\right).
\end{aligned}
\right.
\end{equation*}
The following CLT characterizes the asymptotic distribution of the mutual information in \eqref{eq:MI_expression}. The proof builds on arguments developed in \cite{Bias2022}, with suitable modifications to accommodate sparsity.
\begin{theorem}\label{MI}
	Let $z = -\sigma^{2}$ and $1/4 < \phi \leq 1/2$.
	The CLT for the mutual information, $C_{\mathbf{H}\mathbf{H}^{*}}(\sigma^{2})$, can be expressed as
	\begin{equation*}
		T_{\log}=\frac{	q/\sqrt{N_r}\left[C_{\mathbf{H}\mathbf{H}^{*}}(\sigma^{2})-V(\sigma^2)\right]-\mu_{\log}}{\sigma_{\log}}
		 \xrightarrow{\mathcal D} \mathcal{N}(0, 1),
	\end{equation*}
	where
	\begin{equation*}
		V(\sigma^2) = -\log\det(\sigma^{2}\mathbf{T}(z))
		+ \log\det\!\left(\mathbf{I} + \delta(z)\widetilde{\mathbf{D}}\right)
		- N_t\sigma^{2}\delta(z)\tilde{\delta}(z),
	\end{equation*}
	\begin{align}
		\mu_{\log} 
		&= - \frac{\left(\tilde \nu_4-(\kappa+2)s\right)\sigma^{4}}{2q\sqrt{
		N_r}N_t}
		\operatorname{Tr}\!\left(\mathbf{D}^{2}\mathbf{T}^{2}(z)\right)
		\operatorname{Tr}\!\left(\widetilde{\mathbf{D}}^{2}\widetilde{\mathbf{T}}^{2}(z)\right),
		\nonumber 
	\end{align}
	and
	\begin{align}
		\sigma_{\log}=&\frac{\left(\tilde \nu_4-(\kappa+2)s\right)\sigma^{4}}{N_rN_t}
		\operatorname{Tr}\!\left(\mathbf{D}^{2}\mathbf{T}^{2}(z)\right)
		\operatorname{Tr}\!\left(\widetilde{\mathbf{D}}^{2}\widetilde{\mathbf{T}}^{2}(z)\right)\\ &-\frac{q^2}{N_r}\log\left(1-\frac{z^2}{N_t^2}\operatorname{Tr}(\mathbf D^2 \mathbf T^2(z))\operatorname{Tr}\left(\widetilde{\mathbf{D}}^2\widetilde{\mathbf{T}}^2(z)\right)\right).
		\nonumber
	\end{align}
\end{theorem}

For a target rate $R$,
the outage probability
$P_{\rm out}(R)=\mathbb P\big(C_{\mathbf{H}\mathbf{H}^{*}}(\sigma^{2})<R\big)$
admits the Gaussian approximation
\[
P_{\rm out}(R,\rho)\ \approx\
\Phi\!\left(
\frac{\displaystyle {q}/{\sqrt{N_r}}
	\left(R-V(\sigma^2)\right)
	-\mu_{\log}}{\sigma_{\log}}
\right),
\]
where $\Phi$ is the standard normal Cumulative Distribution Function.


\section{Numerical results}
\subsection{Equality test of two large-scale fading matrices}

To examine the finite-sample performance of the proposed procedures in Section \ref{appl1}, we conduct Monte Carlo experiments based on the sparse MIMO model \eqref{mimo}. The large-scale fading matrices $\mathbf{L}_1$ and $\mathbf{L}_2$ are generated as $N_r \times N_t$ matrices with entries
\[
l_{ij,1} \sim \text{Uniform}(2,4), 
\qquad 
l_{ij,2} = l_{ij,1} - \theta,
\]
where $\theta$ controls the deviation between the two environments: $\theta = 0$ corresponds to the null hypothesis $\mathcal{H}_0$, whereas $\theta \neq 0$ represents the alternative $\mathcal{H}_1$. 
To introduce sparsity, the missing indicators are generated as $b_{ij,1}, b_{ij,2} \sim \rm Bernoulli(1, s)$,
with $s = q^2 / N_t$. Two regimes of $q$ are examined, corresponding to $q = 0.5\, N_t^{1/2}$ (moderate-sparsity regime) and $q = N_t^{1/3}$ (high-sparsity regime). We
examine two  distributions for the small-scale fading terms:\\
(1) the complex Gaussian distribution, where $g_{ij} \sim \mathcal{CN}(0,1)$;\\
(2) the Gamma distribution, where $\sqrt{2}\, g_{ij} + 2 \sim \Gamma(2,1)$.

Tables~\ref{tab1} and~\ref{tab2} report the empirical sizes and powers of the proposed test $T_x$ under different population distributions and system dimensions. 
Specifically, Table~\ref{tab1} corresponds to the setting $q = 0.5 N_t^{1/2}$ while Table~\ref{tab2} considers  $q = N_t^{1/3}$. 
The results are presented for both Gaussian and Gamma populations with various ratios $c_N = N_r / N_t \in \{0.8, 0.5, 0.25\}$ and receive antenna numbers $N_r \in \{200, 400, 800\}$. 
The nominal significance level is fixed at $0.05$.

\begin{table}[bp]
	\centering
	\caption{Empirical sizes and powers of the test $T_x$ when $q=0.5N_t^{1/2}$.}
	\small
	\renewcommand\arraystretch{1.2}
	\setlength{\tabcolsep}{3mm}
	\begin{tabular}{llllll}
		\hline
		Values of $\theta$& \multicolumn{1}{l}{$0$} & \multicolumn{1}{l}{$0.02$} & \multicolumn{1}{l}{$0.05$} & \multicolumn{1}{l}{$0.1$} & \multicolumn{1}{l}{$0.15$}\\
		\hline
		\multicolumn{6}{c}{ Gaussian population } \\
		\hline
		$N_r=200,c_N=0.8$ & \textbf{0.053}  &0.105 &0.448  &0.943 &1 \\
		$N_r=400,c_N=0.8$ & \textbf{0.045}  &0.272 &0.948 & 1 &1  \\
		$N_r=800,c_N=0.8$ & \textbf{0.043}  &0.826 &0.985   &1  &1\\
		\hline
		$N_r=200,c_N=0.5$ & \textbf{0.051}  &0.150 & 0.619 &0.991  &1 \\
		$N_r=400,c_N=0.5$ & \textbf{0.049}  & 0.442 &0.998 &1  &1 \\
		$N_r=800,c_N=0.5$ &\textbf{0.048}  & 0.931& 1  & 1 &1\\
		\hline
		$N_r=200,c_N=0.25$ &\textbf{0.048}  &0.232 &0.885  & 1 &1 \\
		$N_r=400,c_N=0.25$ & \textbf{0.046}  &0.747  &1   &1  &1\\
		$N_r=800,c_N=0.25$ &\textbf{0.051}  &0.986  & 1  &1  &1\\
		\hline
		\multicolumn{6}{c}{Gamma population } \\
		\hline
		$N_r=200,c_N=0.8$ & \textbf{0.051}  &0.081 &0.177 &0.502 &0.849\\
		$N_r=400,c_N=0.8$ & \textbf{0.048}  &0.125&0.532 & 0.984 &1 \\
		$N_r=800,c_N=0.8$ &\textbf{0.053}  &0.350& 0.977 &1  &1 \\
		\hline
		$N_r=200,c_N=0.5$ & \textbf{0.045}  &0.095 &0.251 &0.717 &0.964\\
		$N_r=400,c_N=0.5$ & \textbf{0.048}  &0.199 &0.730 &0.999  &1  \\
		$N_r=800,c_N=0.5$ &\textbf{0.051}  &0.518  &0.999 &1  &1 \\
		\hline
		$N_r=200,c_N=0.25$ &\textbf{0.049}  &0.124 &0.425 &0.940  &1\\
		$N_r=400,c_N=0.25$ &\textbf{0.048}  &0.273 &0.941   &1  &1 \\
		$N_r=800,c_N=0.25$ &\textbf{0.049}  &0.799 & 1  & 1 &1\\
		\hline
	\end{tabular}%
	\label{tab1}%
\end{table}%

\begin{table}[tp]
	\centering
	\caption{Empirical sizes and powers of the test $T_x$ when $q=N_t^{1/3}$.}
	\small
	\renewcommand\arraystretch{1.2}
	\setlength{\tabcolsep}{3mm}
	\begin{tabular}{llllll}
		\hline
		Values of $\theta$& \multicolumn{1}{l}{$0$} & \multicolumn{1}{l}{$0.02$} & \multicolumn{1}{l}{$0.05$} & \multicolumn{1}{l}{$0.1$} & \multicolumn{1}{l}{$0.15$}\\
		\hline
		\multicolumn{6}{c}{ Gaussian population } \\
		\hline
		$N_r=200,c_N=0.8$ & \textbf{0.048}  &0.118 &0.400    &0.898 &0.998 \\
		$N_r=400,c_N=0.8$ & \textbf{0.045}  &0.271&0.832 &1 &1  \\
		$N_r=800,c_N=0.8$ & \textbf{0.049}  &0.567 &0.999 & 1 &1\\
		\hline
		$N_r=200,c_N=0.5$ & \textbf{0.047}  &0.171 &0.622 &0.994 &1  \\
		$N_r=400,c_N=0.5$ & \textbf{0.049}  &0.406 &0.973 &1  &1 \\
		$N_r=800,c_N=0.5$ &\textbf{ 0.053}  &0.813  &0.998 &1  &1\\
		\hline
		$N_r=200,c_N=0.25$ &\textbf{0.045}  &0.342 &0.940 &0.998  &1 \\
		$N_r=400,c_N=0.25$ & \textbf{0.050}  &0.709   &1  & 1 &1\\
		$N_r=800,c_N=0.25$ &\textbf{0.049}  &0.994  &1   & 1 &1\\
		\hline
		\multicolumn{6}{c}{Gamma population } \\
		\hline
		$N_r=200,c_N=0.8$ & \textbf{0.046}  &0.096 &0.199 & 0.498 & 0.795\\
		$N_r=400,c_N=0.8$ & \textbf{0.047}  & 0.135& 0.436 &0.925  &0.998 \\
		$N_r=800,c_N=0.8$ &\textbf{0.045}  & 0.282& 0.844&1  &1 \\
		\hline
		$N_r=200,c_N=0.5$ & \textbf{0.048}  & 0.124&0.307 & 0.726 &0.977\\
		$N_r=400,c_N=0.5$ & \textbf{0.052}  &0.177 & 0.660&0.991 &1  \\
		$N_r=800,c_N=0.5$ & \textbf{0.050}  &0.422 &0.978  & 1 &1  \\
		\hline
		$N_r=200,c_N=0.25$ &\textbf{0.049}  &0.189 &0.610  & 0.984 &1\\
		$N_r=400,c_N=0.25$ &\textbf{0.048}  &0.372  &0.946 & 1 &1 \\
		$N_r=800,c_N=0.25$ &\textbf{0.050}  &0.730 &0.998   &1  &1\\
		\hline
	\end{tabular}%
	\label{tab2}%
\end{table}

For all configurations, the empirical sizes (at $\theta=0$) are close to the nominal level, indicating that the test is well calibrated. 
As $\theta$ increases, the empirical power rapidly approaches one, demonstrating the strong discriminating ability of $T_x$ under the alternative. 
Moreover, the power improves with larger $N_r$, and—importantly—with smaller $c_N$ (i.e., larger $N_t$ for fixed $N_r$). 
This agrees with the theoretical expectation that stronger signal or higher dimensionality enhances detection capability.

Figure~\ref{fig1} compares the empirical and theoretical power functions of $T_x$ under the Gaussian population for both sparsity regimes. 
In the moderate-sparsity regime, the empirical curves (markers) closely follow the theoretical predictions (lines) across different scaling parameters $s \in (0,1)$. 
A similar trend is observed in the high-sparsity regime, where the power curves rapidly converge to one as $\theta$ grows. 
Overall, the excellent agreement between simulation and theory confirms the validity of the asymptotic power analysis and the robustness of the proposed test statistic across different sparsity levels and population models.

\begin{figure}[tp]
	\centering
	\subfigure[Moderate-sparsity regime]{\includegraphics[width=0.48\linewidth]{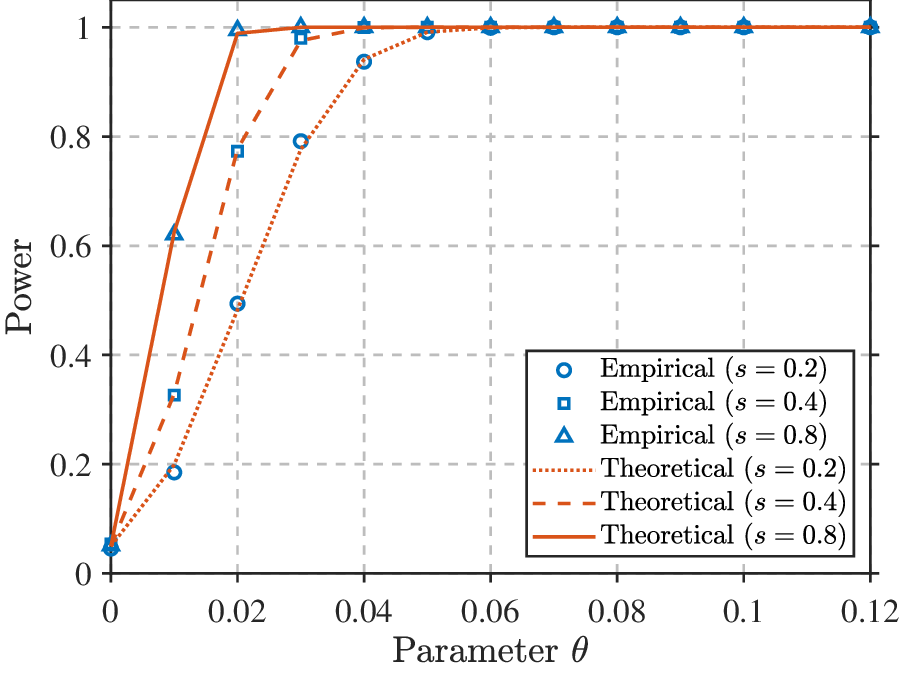}}
	\subfigure[ High-sparsity regime]{\includegraphics[width=0.48\linewidth]{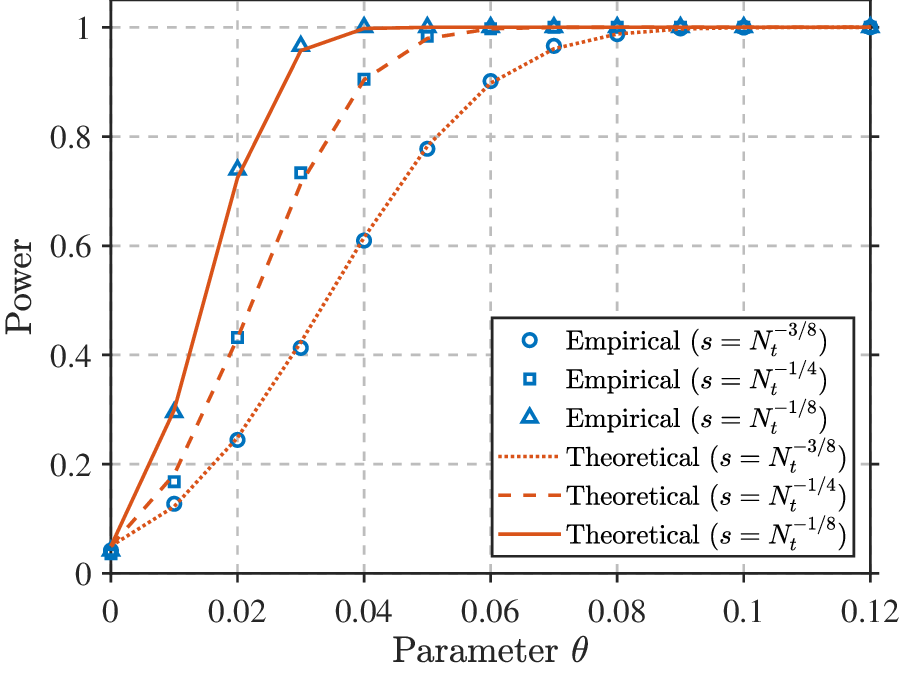}}
	\caption{Comparison of empirical and theoretical powers of the test $T_x$ under the Gaussian distribution. }
	\label{fig1}
\end{figure}

\subsection{Outage probability analysis}

To validate the CLT established in Theorem~\ref{MI}, we conduct Monte Carlo simulations focusing on two aspects: 
(i) the empirical distribution of the normalized mutual information, examined through histogram; and 
(ii) the accuracy of the theoretical outage probability approximation. 

The MIMO channel is generated according to~\eqref{eq:channel_model}, where the diagonal entries are drawn as $d_i \sim \text{Uniform}(1,2)$ for $i=1,\ldots,N_r$ and $\tilde{d}_j \sim \text{Uniform}(1,2)$ for $j=1,\ldots,N_t$. 
The random matrix $\mathbf{X}$ consists of i.i.d. entries $x_{ij} \sim \mathcal{CN}(0,1)$. Also, two cases of $q$ are considered, corresponding to  $q = 0.5 N_t^{1/2}$ (moderate-sparsity regime) and $q = N_t^{1/3}$ (high-sparsity regime).
Unless otherwise specified, we set $(N_r, N_t) = (512, 1024)$ and average the results over $2,000$ independent channel realizations.

Figure~\ref{fig:histqq} compares the empirical and theoretical distributions of the normalized statistic $T_{\log}$ for $\sigma=1$ under two sparsity regimes. 
In the moderate-sparsity case of $q=0.5N_t^{1/2}$, the histogram of $T_{\log}$ nearly coincides with the standard normal density.
A similar observation holds for the high-sparsity regime of  $q=N_t^{1/3}$.
Figure~\ref{fig:outage} shows the empirical and theoretical outage probabilities versus the transmission rate~$R$ for different signal-to-noise ratios (SNRs). 
The SNR (in dB) is related to the noise variance by 
$\text{SNR}_{\text{dB}} = 10 \log_{10} (1/\sigma^2)$, i.e., $\sigma^2 = 10^{-\text{SNR}_{\text{dB}}/10}$. 
Three SNR values are considered in the simulations, namely $0$, $0.2$, and $0.4~\text{dB}$. 
The analytical curves closely match the Monte Carlo results in both sparsity regimes, and a clear rightward shift is observed with increasing SNR, indicating improved achievable rates. 

Overall, the results in Figures~\ref{fig:histqq} and~\ref{fig:outage} validate the proposed asymptotic analysis: 
the Gaussianity of $T_{\log}$ confirms the CLT characterization, while the agreement in outage behavior demonstrates the accuracy of the theoretical approximation in practical finite-dimensional settings.

\begin{figure}[bp]
	\centering
	\subfigure[Moderate-sparsity regime]{\includegraphics[width=0.48\linewidth]{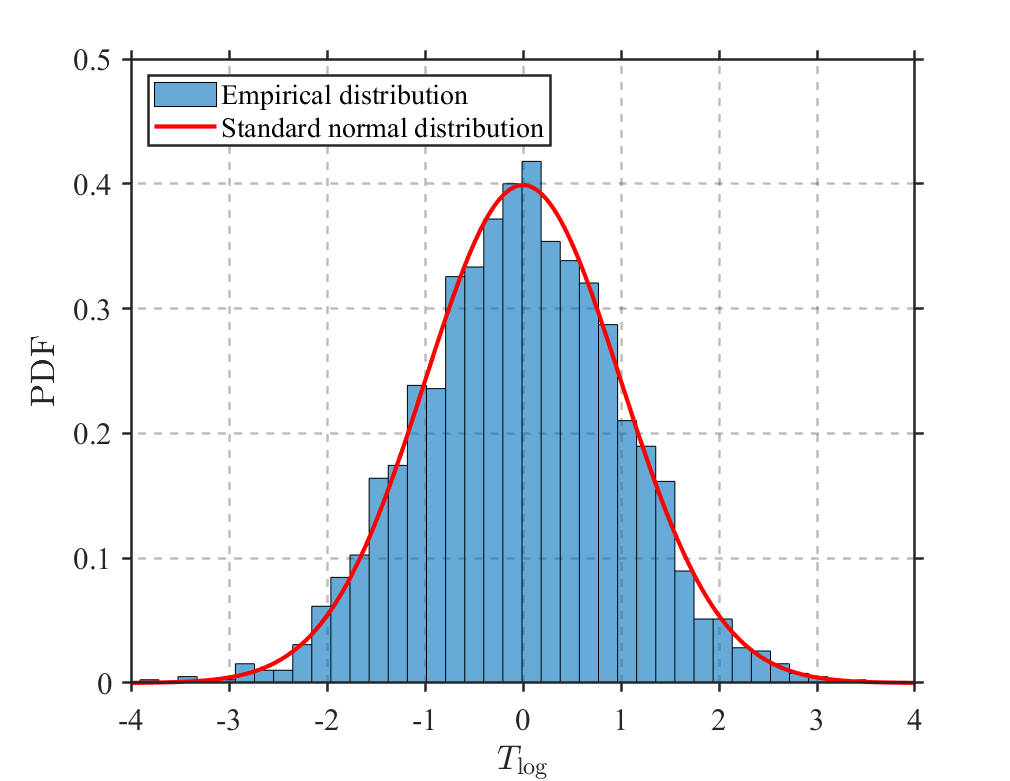}}
    \subfigure[ High-sparsity regime]{\includegraphics[width=0.48\linewidth]{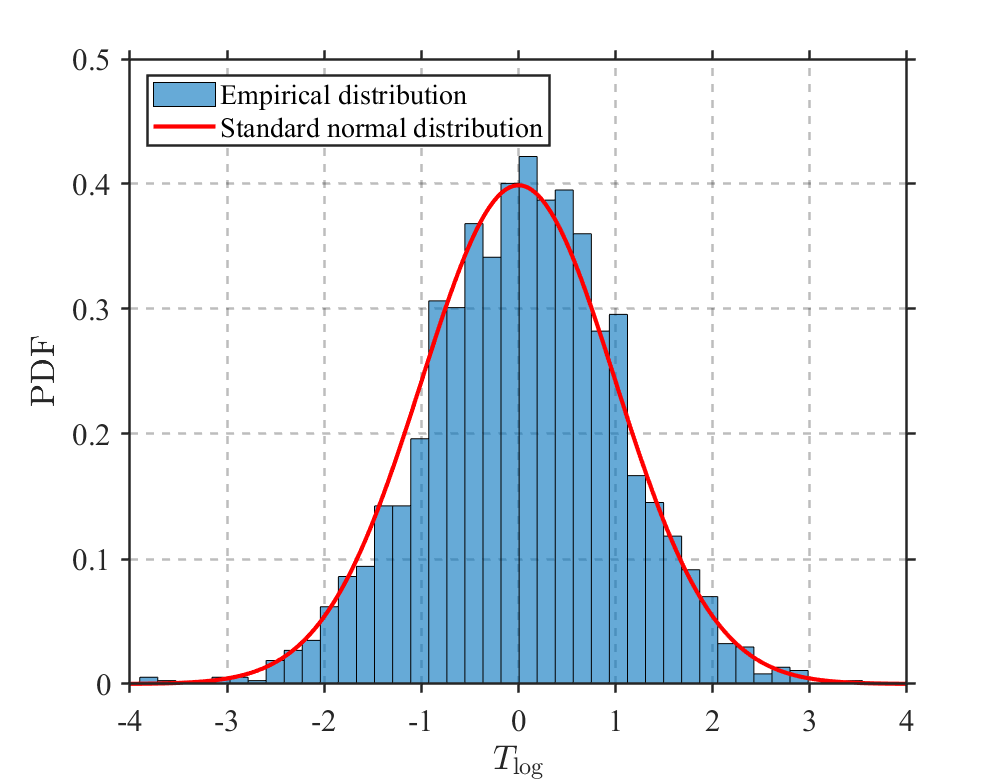}}
	\caption{Comparison of empirical and theoretical distributions of $T_{\log}$. }
	\label{fig:histqq}
\end{figure}

\begin{figure}[htbp]
	\centering
	\subfigure[Moderate-sparsity regime]{\includegraphics[width=0.49\linewidth]{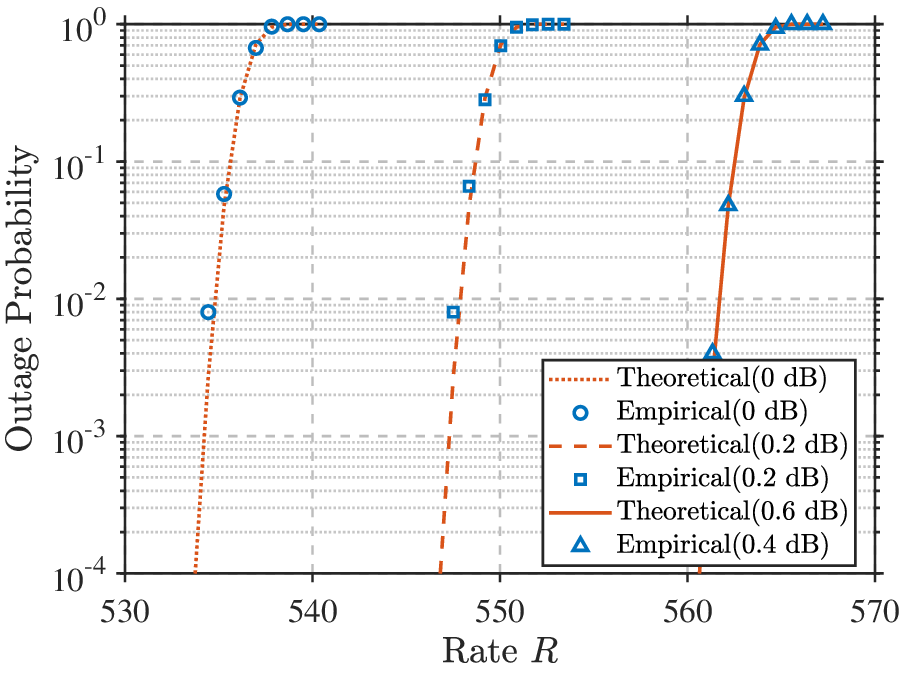}}
    \subfigure[ High-sparsity regime]{\includegraphics[width=0.49\linewidth]{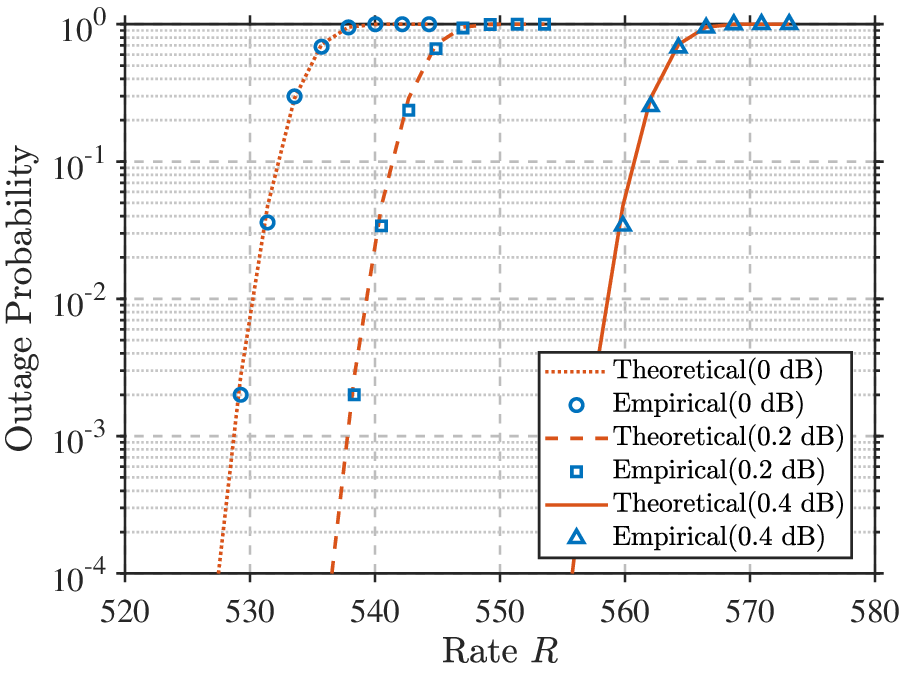}}
	\caption{Comparison of empirical and theoretical outage probability. }
	\label{fig:outage}
\end{figure}


\begin{appendix}
\section{Proofs}\label{app1}

\subsection{Proof of Theorem \ref{first-order}} \label{proof-th1}
We prove the two assertions separately.
	
\textbf{Proof of (1).}  
Our model satisfies the assumptions in \cite{Hachem2007}, including uniformly bounded variances and appropriate moment conditions. Consequently, part (a) follows from their Theorem~2.4 therein. The argument depends only on the second-order structure of the entries and applies to both dense and sparse settings; we therefore omit the routine details. 
		
\textbf{Proof of (2).} 
We consider $x_{ij} = b_{ij} w_{ij}/\sqrt{s}$, where $b_{ij} \sim \operatorname{Bernoulli}(1,s)$ and $w_{ij}$ are centered with $\mathbb{E}|w_{ij}|^2 = \sigma_{ij}^2$ and finite $4+\epsilon$ moments. Unlike the dense model in \cite{Hachem2007}, sparsity introduces a significant challenge: in the high-sparsity regime ($s \to 0$), the $4+\epsilon$ moments of $x_{ij}$ may diverge, and classical resolvent estimates for quadratic forms no longer apply directly.
When $s$ is fixed and bounded away from zero (moderate sparsity), classical techniques can be adapted. Our main contribution is the extension of these results to the high-sparsity regime, which relies on refined resolvent-based bounds for quadratic forms under diverging entry-wise moments.



		To address this difficulty, we introduce intermediate diagonal matrices  $\mathbf P_n$ and  $\widetilde{\mathbf P}_n$, and decompose the analysis into two differences:
		\[
		\frac{1}{p}\operatorname{Tr}(\mathbf Q_n-\mathbf P_n)
		\quad\text{and}\quad
		\frac{1}{p}\operatorname{Tr}(\mathbf P_n-\mathbf T_n),
		\]
       and similarly,
        \[
		\frac{1}{n}\operatorname{Tr}(\widetilde{\mathbf Q}_n-\widetilde{\mathbf P}_n)
		\quad\text{and}\quad
		\frac{1}{n}\operatorname{Tr}(\widetilde{\mathbf P}_n-\widetilde{\mathbf T}_n).
		\]
        The matrices $\mathbf P_n$ and $\widetilde{\mathbf P}_n$ are diagonal, defined by
        \begin{align*}
            p_i&\triangleq[\mathbf P_n]_{ii}=-\frac{1}{z\Bigl(1+\tfrac{1}{n}\operatorname{Tr}(\widetilde{\mathbf \Sigma}_i\widetilde{\mathbf Q}_n)\Bigr)}, \quad i=1,2,\cdots,p,\\
            \tilde p_j&\triangleq[\widetilde{\mathbf P}_n]_{jj}=-\frac{1}{z\Bigl(1+\tfrac{1}{n}\operatorname{Tr}({\mathbf \Sigma}_j{\mathbf Q}_n)\Bigr)}, \quad j=1,2,\cdots,n.
        \end{align*}

		For notational convenience, we omit the subscript $n$ in what follows.

		\emph{Step 1: ${p}^{-1}\operatorname{Tr}(\mathbf Q_n-\mathbf P_n)\to 0$ and ${n}^{-1}\operatorname{Tr}(\widetilde{\mathbf Q}_n-\widetilde{\mathbf P}_n) \to 0$.} 
        For the first term, we employ leave-one-out techniques. Denote by $\mathbf y_j$ the $j$-th column of $\mathbf Y$, and let $\mathbf Y^{(j)}$ be the $p\times (n-1)$ matrix obtained by deleting the $j$th column. Define the resolvents:
		\[
		\mathbf Q^{(j)}(z)=(\mathbf Y^{(j)}{\mathbf Y^{(j)}}^*-z\mathbf I)^{-1}, 
		\qquad 
		\widetilde{\mathbf Q}^{(j)}(z)=({\mathbf Y^{(j)}}^*\mathbf Y^{(j)}-z\mathbf I)^{-1}.
		\]
		Let $\widetilde{\mathbf \Sigma}_i^{(j)}$ denote the $(n-1)\times(n-1)$  matrix obtained by removing row $j$ and column $j$ from $\widetilde{\mathbf \Sigma}_i$, and let $\mathbf P^{(j)}$ be defined analogously with diagonal entries
		\[
		[\mathbf P^{(j)}]_{ii}=-\frac{1}{z\Bigl(1+\tfrac{1}{n}\operatorname{Tr}(\widetilde{\mathbf \Sigma}_i^{(j)}\widetilde{\mathbf Q}^{(j)})\Bigr)}.
		\]

        Introduce
		\[
		\omega_j=\mathbf y_j^*\mathbf Q^{(j)}\mathbf y_j,
		\qquad 
		\hat{\omega}_j=\mathbf y_j^*\mathbf P \mathbf U \mathbf Q^{(j)}\mathbf y_j.
		\]
        Let $\mathbf U_n$ and $\widetilde{\mathbf U}_n$ denote  deterministic diagonal sequences satisfying $\sup_n \max_{i}|[\mathbf U_n]_{ii}|\le K\le \infty$ and $\sup_n \max_{j}|[\widetilde{\mathbf U}_n]_{jj}|\le \tilde K\le \infty$. 
		Using the resolvent identity and formulas  from \cite[(6.5), (B.17)–(B.18)]{Hachem2007}, one obtains
		\begin{align}
			\frac{1}{p}\operatorname{Tr}(\mathbf P-\mathbf Q)\mathbf U
			&=\frac{1}{p}\operatorname{Tr}\mathbf Q(\mathbf Q^{-1}-\mathbf P^{-1})\mathbf P\mathbf U \notag\\
			&=\frac{1}{p}\sum_{j=1}^n \frac{\hat{\omega}_j}{1+\omega_j}-\frac{1}{p}\operatorname{Tr}\left(\mathbf P^{-1}+z\mathbf I\right)\mathbf P \mathbf U \mathbf Q \notag\\
			&=-\frac{1}{p}\sum_{j=1}^n z[\widetilde{\mathbf Q}]_{jj}(z)\left(\hat{\omega}_j-\frac{1}{n}\operatorname{Tr}\mathbf \Sigma_j \mathbf P \mathbf U \mathbf Q\right)\notag\\
			&=-\frac{1}{p}\sum_{j=1}^n z[\widetilde{\mathbf Q}]_{jj}(z)\left(\chi_{j,1}+\chi_{j,2}\right),  \label{ev1}
		\end{align}
		where
		\begin{align*}
			\chi_{j,1}&=\mathbf y_j^* \mathbf P^{(j)}\mathbf U \mathbf Q^{(j)}\mathbf y_j-\frac{1}{p}\operatorname{Tr}\left(\mathbf \Sigma_j \mathbf P(z)\mathbf U\mathbf Q(z)\right), \\
			\chi_{j,2}&=\mathbf y_j^*\mathbf P\left(\mathbf P^{(j)^{-1}}-\mathbf P^{-1}\right)\mathbf P^{(j)}\mathbf U\mathbf Q^{(j)}\mathbf y_j.
		\end{align*}
        Now decompose $\chi_{j,1}$ as
		\begin{align*}
			\chi_{j,1}=&\;\mathbf y_j^*\mathbf P^{(j)}\mathbf U \mathbf Q^{(j)}\mathbf y_j-\frac{1}{p}\operatorname{Tr}\left(\mathbf \Sigma_j \mathbf P^{(j)}(z)\mathbf U\mathbf Q^{(j)}(z)\right)\\
			&+\frac{1}{p}\operatorname{Tr}\mathbf \Sigma_j\mathbf P^{(j)}(z)\mathbf U\left(\mathbf Q^{(j)}(z)-\mathbf Q(z)\right)
			+\frac{1}{p}\operatorname{Tr}\mathbf \Sigma_j\left(\mathbf P^{(j)}(z)-\mathbf P(z)\right)\mathbf U\mathbf Q(z).
		\end{align*}
		From Lemma \ref{quad_ineq} and Minkowski’s inequality, it follows that
		\[
		\bbe|\chi_{j,1}|^{2+\epsilon/2}\leq \frac{K}{q^{2+\epsilon/2}}.
		\]
		For $\chi_{j,2}$, using \cite[(B.20)]{Hachem2007} and Assumption \ref{ass1}, we obtain
		\[
		\bbe|\chi_{j,2}|^{2+\epsilon/2}\leq \frac{K}{n^{2+\epsilon/2}}\bbe\|\mathbf y_j\|^{4+\epsilon}
		\leq \frac{K}{n^{2+\epsilon}}.
		\]
		Inserting these bounds into \eqref{ev1} and applying Minkowski’s inequality yields
		\[
		\mathbb E\left|\frac{1}{p}\sum_{j=1}^n z[\widetilde{\mathbf Q}]_{jj}(z)(\chi_{j,1}+\chi_{j,2})\right|^{2+\epsilon/2}
		\leq \frac{K}{q^{2+\epsilon/2}}.
		\]
        
		Therefore,
		\[
		\mathbb E\left|\frac{1}{p}\operatorname{Tr}(\mathbf P-\mathbf Q)\mathbf U\right|^{2+\epsilon/2}
		\leq \frac{K}{q^{2+\epsilon/2}}.
		\]

      By Markov's inequality and the Borel--Cantelli lemma,
    \[
    \frac{1}{p} \operatorname{Tr}(\mathbf{Q}_n - \mathbf{P}_n) \to 0 \quad \text{almost surely}.
    \]
    The same holds for ${n}^{-1} \operatorname{Tr}(\widetilde{\mathbf{Q}}_n - \widetilde{\mathbf{P}}_n)$ by symmetry.

	\emph{Step 2: ${p}^{-1}\operatorname{Tr}(\mathbf P_n-\mathbf T_n) \to 0$ and ${n}^{-1}\operatorname{Tr}(\widetilde{\mathbf P}_n-\widetilde{\mathbf T}_n) \to 0$.}
    We adapt the arguments from  \cite[Section~6.2]{Hachem2007}. 
    Consider the domain
    \[
    D := \{z\in\mathbb C^{+}: |\Im z|\ge v_0,\ |z|\le M\},
    \]
    with fixed $v_0 > 0$, $M < \infty$. On $D$, all resolvents satisfy
    \[
    \|\mathbf T(z)\|,\ \|\mathbf P(z)\|,\ \|\widetilde{\mathbf T}(z)\|,\ \|\widetilde {\mathbf P}(z)\|
    \le \frac{1}{|\Im z|}.
    \]

        From the identity $
        \mathbf P-\mathbf T=\mathbf P(\mathbf T^{-1}-\mathbf P^{-1})\mathbf T
        $, we get  
        $$
        \frac{1}{p}\operatorname{Tr}(\mathbf P-\mathbf T)=\frac{1}{p} \sum_{i} (\frac{1}{t_i}-\frac{1}{p_i})t_ip_i,
        $$
        where $|t_i(z)p_i(z)|\leq 1/|\Im z|^2$. Hence, by Minkowski's inequality,
        \begin{align} \label{P-T}
            \left\|\frac{1}{p}\operatorname{Tr}(\mathbf P-\mathbf T)\right\|_{2+\epsilon} \leq \frac{1}{|\Im z|^2}\sup_i \left\|\frac{1}{t_i}-\frac{1}{p_i}\right\|_{2+\epsilon}.
        \end{align}
        From the definitions  of $t_i(z)$ and of $p_i(z)$, we obtain
        $$
        \frac{1}{t_i(z)}-\frac{1}{p_i(z)}=\frac{z}{n} \operatorname{Tr}\widetilde{\mathbf \Sigma}_i(\widetilde{\mathbf T}-\widetilde{\mathbf P})+\frac{z}{n} \operatorname{Tr}\widetilde{\mathbf \Sigma}_i(\widetilde{\mathbf P}-\widetilde{\mathbf Q}).
        $$
        By Step~1 and $\|\widetilde{\mathbf\Sigma}_i\|\le\sigma_{\max}^{2}$, 
    \[
    \sup_i \Bigl\|\frac{z}{n}\Tr\widetilde{\mathbf \Sigma}_i(\widetilde{\mathbf P}-\widetilde{\mathbf Q})\Bigr\|_{2+\epsilon/2}\le \frac{K|z|}{q}.
     \]

        Similarly,  rewriting
        $   \widetilde{\mathbf T}-\widetilde{\mathbf P}=\widetilde{\mathbf T}(\widetilde{\mathbf P}^{-1}-\widetilde{\mathbf T}^{-1})\widetilde{\mathbf P}$, and repeating the argument leading to~\eqref{P-T}, we obtain
        \begin{align}\label{t-p}
            \left\|\frac{1}{t_i(z)}-\frac{1}{p_i(z)}\right\|_{2+\epsilon/2}\leq  \frac{K_1|z|\sigma_{\max}^2}{|\Im z|^2}\sup_j \left\| \frac{1}{\tilde t_j(z)}-\frac{1}{\tilde p_j(z)}\right\|_{2+\epsilon/2}+\frac{K_2|z|}{q}.
        \end{align}
        Likewise, for each $1 \leq j \leq n$, we have 
        \begin{align}\label{tilde t-p}
            \left\|\frac{1}{\tilde t_j(z)}-\frac{1}{\tilde p_j(z)}\right\|_{2+\epsilon/2}\leq  \frac{K_1^\prime|z|\sigma_{\max}^2}{|\Im z|^2}\sup_k \left\| \frac{1}{t_k(z)}-\frac{1}{p_k(z)}\right\|_{2+\epsilon/2}+\frac{K_2^\prime|z|}{q}.
        \end{align}
        Combining inequalities \eqref{t-p} and \eqref{tilde t-p}, this gives
        \[
        \sup_i 
        \Bigl\|\frac1{t_i(z)}-\frac1{p_i(z)}\Bigr\|_{2+\epsilon/2}
        \le \frac{K}{q}.
        \]
        Together with \eqref{P-T}, we conclude that for each $z \in D$,
		\[
		\mathbb E\left|\frac{1}{p}\operatorname{Tr}(\mathbf P-\mathbf T)\right|^{2+\epsilon/2}
		\leq \frac{K}{q^{2+\epsilon/2}},
		\]
        and similarly
        \[
		\mathbb E\left|\frac{1}{n}\operatorname{Tr}(\widetilde{\mathbf P}-\widetilde{\mathbf T})\right|^{2+\epsilon/2}
		\leq \frac{K^\prime}{q^{2+\epsilon/2}}.
		\]

        \emph{Step 3. End of the proof.} 
        Combining Steps~1 and~2, we have
    \[
    \mathbb{E} \left| \frac{1}{p} \operatorname{Tr}(\mathbf{Q} - \mathbf{T}) \right|^{2+\epsilon/2} \leq \frac{K}{q^{2+\epsilon/2}}, \quad
    \mathbb{E} \left| \frac{1}{n} \operatorname{Tr}(\widetilde{\mathbf{Q}} - \widetilde{\mathbf{T}}) \right|^{2+\epsilon/2} \leq \frac{K'}{q^{2+\epsilon/2}}.
    \]

    By Markov's inequality, convergence in probability holds:
    \[
    \frac{1}{p} \operatorname{Tr}(\mathbf{Q} - \mathbf{T}) \xrightarrow{\mathcal{P}} 0, \quad
    \frac{1}{n} \operatorname{Tr}(\widetilde{\mathbf{Q}} - \widetilde{\mathbf{T}}) \xrightarrow{\mathcal{P}} 0.
    \]
    
    Moreover, if $q \to \infty$ sufficiently fast so that $\sum_n q^{-2-2+\epsilon/2} < \infty$, then Borel--Cantelli yields almost sure convergence along subsequences. Finally, by analyticity and the fact that the resolvent traces form a normal family on any compact subset of $\mathbb{C} \setminus \mathbb{R}^+$ (see \cite[Section~6.3]{Hachem2007}), the convergence extends uniformly on such sets.

		
			

\subsection{Outline of the proof of Theorem \ref{th1}}

To characterize the asymptotic distribution in Theorem \ref{th1}, we first introduce an auxiliary lemma that provides the key quantities $U_j(z)$ and $\tilde{U}_i(z)$. The proof is provided in the supplementary material.
\begin{lemma}\label{mean-eq}
    Consider a variance profile $(\sigma_{ij}^2)$ satisfying Assumptions \ref{ass2}-\ref{ass3} , and let $t_i$, $\tilde t_j$ be as defined in Theorem \ref{first-order}.
    Fix $z \in \mathbb{C} \setminus \mathbb{R}^+$.

    \begin{enumerate}
        \item For $1 \leq j \leq n$, the quantity $U_j(z)$ admits the representation
        \[
        U_j(z) = \frac{1}{z^2 \tilde{t}_j^2(z)} \breve{y}_{j,j}(z),
        \]
        where $\breve{y}_{j,j}(z)$ is determined by the $n$-dimensional linear system
        \begin{align}
            \breve{y}_{l,j}(z) = \sum_{i \neq j}^{n} a_{li}(z,z) \breve{y}_{i,j}(z) + n a_{lj}(z,z), \quad 1 \leq l,j \leq n, \label{Uj}
        \end{align}
        in the unknowns $\{\breve{y}_{l,j}(z): 1 \leq l,j \leq n\}$. This system has a unique solution for all sufficiently large $n$.

        \item Similarly, for $1 \leq i \leq p$, the quantity $\tilde{U}_i(z)$ satisfies
        \[
        \tilde{U}_i(z) = \frac{1}{z^2 t_i^2(z)} \tilde{y}_{i,i}(z),
        \]
        where $\tilde{y}_{i,i}(z)$ is obtained from the $p$-dimensional linear system
        \begin{align}
            \tilde{y}_{i,m}(z) = \sum_{l \neq m}^{p} \acute{a}_{il}(z) \tilde{y}_{l,m}(z) + n \acute{a}_{im}(z), \quad 1 \leq i,m \leq p, \label{Ui}
        \end{align}
        with coefficients
        \[
        \acute{a}_{im}(z) = \frac{1}{n^2} \frac{\operatorname{Tr} \left( \widetilde{\mathbf{\Sigma}}_i \widetilde{\mathbf{T}}(z) \widetilde{\mathbf{\Sigma}}_m \widetilde{\mathbf{T}}(z) \right)}{\left(1 + \frac{1}{n} \operatorname{Tr} \left( \widetilde{\mathbf{\Sigma}}_i \widetilde{\mathbf{T}}(z) \right) \right)^2}.
        \]
        This system also admits a unique solution for all sufficiently large $n$.
    \end{enumerate}
\end{lemma}

Now proceed with the main proof.  Since Corollary~\ref{coro1} is an immediate consequence of Theorem~\ref{th1}, it suffices to prove the latter.

The first step is to replace the entries of $\mathbf X_n$ by appropriately truncated and centralized versions.   
By Assumption~\ref{ass1}, we may choose $\delta_n \downarrow 0$ with $\delta_n q \uparrow \infty$ such that
\begin{equation}\label{linderberg}
\frac{1}{(\delta_n q)^4}
\sum_{i,j} \mathbb E\!\left[|w_{ij}|^{4}\,
I(|w_{ij}| \ge \delta_n q)\right]
\rightarrow 0,
\end{equation}
for $n^{\phi}\le q \le n^{1/2}$ and $\phi>1/4$.  
The sequence $\delta_n$ may be taken to decay arbitrarily slowly.  
This Lindeberg-type condition follows from the existence of moments of order $8+\epsilon$ and ensures that the truncation does not alter the limiting behavior.

Let $\widehat{\mathbf S}_n = (ns)^{-1} (\mathbf B\circ \widehat{\mathbf W}_n)(\mathbf B\circ \widehat{\mathbf W}_n)^*$,  
where $\widehat{\mathbf W}_n$ has entries  
$\hat w_{ij} = w_{ij}I(|w_{ij}|<\delta_n q)$.  
Then
\begin{align}
\mathbb P(\mathbf S_n \neq \widehat{\mathbf S}_n)
&\le \sum_{i,j}\mathbb P(|w_{ij}|\ge \delta_n q)  \notag \\
&\le \frac{1}{(\delta_n q)^4}
    \sum_{i,j}\mathbb E\!\left[|w_{ij}|^{4}I(|w_{ij}| \ge \delta_n q)\right]= o(1).  \notag
\end{align}

Define 
\[
\widehat{\mathbf X}_n = s^{-1/2}(\mathbf B\circ \widehat{\mathbf W}_n),
\qquad
\widetilde{\mathbf X}_n = s^{-1/2}(\mathbf B\circ \widetilde{\mathbf W}_n),
\]
where the truncated variables are
\[
\hat w_{ij}=w_{ij}I(|w_{ij}|<\delta_n q),
\qquad
\tilde w_{ij}=\hat w_{ij}-\mathbb E\hat w_{ij}.
\]
Then
\[
\tilde x_{ij}
=s^{-1/2}(b_{ij}\tilde w_{ij})
=s^{-1/2}b_{ij}\big(\hat w_{ij}-\mathbb E(\hat w_{ij})\big).
\]

Let $\widehat{L}^c_n(f)$ and $\widetilde{L}^c_n(f)$ denote the analogues of $L^c_n(f)$ obtained by replacing 
$\mathbf S_n$ with $\widehat{\mathbf S}_n$ and $\widetilde{\mathbf S}_n$, respectively.  
Let $\lambda_i^{\mathbf A}$ be the $i$-th largest eigenvalue of a Hermitian matrix $\mathbf A$.  
Using the argument  in \cite[Corollary~A.42]{bai2010} and Cauchy--Schwarz's inequality, we obtain
\begin{align*}
\mathbb E\big|\widetilde{L}^c_n(f)-\widehat{L}^c_n(f)\big|
&\le \frac{K_f q}{\sqrt p}\sum_{i=1}^p 
   \mathbb E\Big|\lambda_i(\widehat{\mathbf S}_n) -\lambda_i(\widetilde{\mathbf S}_n)\Big| \\
&\le \frac{K_f q}{n\sqrt p}\sum_{i=1}^n 
   \mathbb E\Big|\lambda_i(\widehat{\mathbf X}_n\widehat{\mathbf X}_n^{*})
               -\lambda_i(\widetilde{\mathbf X}_n\widetilde{\mathbf X}_n^{*})\Big| \\
&\le \frac{K_f q}{n\sqrt p}
      \mathbb E\!\left[
      \operatorname{Tr}(\widetilde{\mathbf X}_n-\widehat{\mathbf X}_n)
      (\widetilde{\mathbf X}_n-\widehat{\mathbf X}_n)^{*}
      \cdot 2\Big(
      \operatorname{Tr}\widehat{\mathbf X}_n\widehat{\mathbf X}_n^{*}
      +\operatorname{Tr}\widetilde{\mathbf X}_n\widetilde{\mathbf X}_n^{*}\Big)
      \right]^{1/2},
\end{align*}
where   $K_f$ is a bound on $|f^\prime(x)|$.

Since $\tilde w_{ij} = \hat w_{ij}-\mathbb E\hat w_{ij}$,
$|\hat w_{ij}-\tilde w_{ij}| = |\mathbb E\hat w_{ij}|$. 
From the truncation condition~\eqref{linderberg},
\[
\sum_{i,j}|\mathbb E\hat w_{ij}|^2
\le \frac{1}{\delta_n^6 q^6}
   \sum_{i,j}\{\mathbb E[|w_{ij}|^4I(|w_{ij}|\ge\delta_nq)]\}^2
= o(1).
\]
Thus
\begin{align*}
	& \frac{q}{n\sqrt{p}}\bbe \operatorname{Tr}(\widetilde{\mathbf{X}}_n-\widehat{\mathbf{X}}_n)(\widetilde{\mathbf{X}}_n-\widehat{\mathbf{X}}_n)^*\\
	\leq &\frac{q}{n\sqrt{p}}\sum_{i, j} \bbe\left|\hat{w}_{i j}-\tilde{w}_{i j}\right|^2 = \frac{q}{n\sqrt{p}}\sum_{i, j}|\bbe \hat w_{ij}|^2=o(1).
\end{align*}
Furthermore,
\[
\mathbb E\operatorname{Tr}\widehat{\mathbf X}_n\widehat{\mathbf X}_n^{*}
\le \sum_{i,j}\mathbb E|\hat w_{ij}|^2
\le Knp,
\qquad
\mathbb E\operatorname{Tr}\widetilde{\mathbf X}_n\widetilde{\mathbf X}_n^{*}
\le \sum_{i,j}\mathbb E|\tilde w_{ij}|^2 \le Knp.
\]
Therefore,
\[
\mathbb E\big|\widetilde{L}^c_n(f)-\widehat{L}^c_n(f)\big|\to0.
\]

Let $\hat\sigma_{ij}^2$ be the variance of $\hat x_{ij}$ (equivalently, the variance of $\hat w_{ij}$).  From truncation, we know that
\[
\big|\sigma_{ij}^2-\hat\sigma_{ij}^2\big|
\le \frac{2}{\delta_n^2 q^2}
   \mathbb E[|w_{ij}|^4I(|w_{ij}|\ge\delta_n q)]\
=o(1)
\]
uniformly in $i,j$.  
Thus the truncation and centering step does not affect the variance profile in the limit. 
For notational simplicity, we henceforth write $w_{ij}$ for $\tilde w_{ij}$, which satisfy the strengthened conditions 
\begin{align*}
    &|w_{ij}|\le \delta_n q,\quad
\mathbb E w_{ij}=0,\\
&\mathbb E|w_{ij}|^{2}=\sigma_{ij}^{2}+o(1) \quad \text{uniformly in } i,j,\\
&\mathbb E|w_{ij}|^{8+\epsilon}
=\tilde\nu_{8+\epsilon}|\sigma_{ij}|^{8+\epsilon}+o(1) \quad \text{uniformly in } i,j,
\end{align*}
where $\tilde{\nu}_{8+\epsilon}$ denotes a finite constant depending on the $8+\epsilon$ moment of the underlying distribution. Moreover, $\delta_n$ satisfies $\lim _{n \rightarrow \infty}1/(\delta_n q)^{4} \sum_{i,j} \mathbb{E}\left|w_{i j}\right|^{4}I\left(\left|w_{i j}\right| \geq \delta_n q\right)=0$,  $\delta_n \downarrow 0$, and $\delta_n q \uparrow \infty$.
All subsequent analysis is carried out under these bounded-moment conditions.

		To link the LSS with the resolvent formulation, recall that for any analytic $f$ and cumulative distribution function $Q$,
		$$
		\int f(x) \mathrm d Q(x)=-\frac{1}{2 \pi i} \oint f(z) m_Q(z) \mathrm d z,
		$$
		where the contour integral is taken along any positively oriented closed path enclosing the support of $Q$ on which $f$ is analytic.
		Define 
		\[
		m_n^0(z) = \frac{1}{p}\sum_{i=1}^p t_i(z),
		\qquad 
		\underline{m}_n^0(z) = \frac{1}{n}\sum_{j=1}^n \tilde t_j(z),
		\]
		and let
		\[
		M_n(z)
		= \sqrt{p}\,q\big[m_n(z) - m_n^0(z)\big]
		= \frac{nq}{\sqrt{p}}\big[\underline m_n(z) - \underline m_n^0(z)\big].
		\]
		The proof of Theorem~\ref{th1} thus reduces to establishing the limiting behavior of the stochastic process $M_n(z)$.
		More precisely, we consider a truncated version $\widehat{M}_n(\cdot)$ of $M_n(\cdot)$, viewed as a random process defined on a contour $\mathcal{C}$ in the complex plane, described as follows.
		Let $v_0>0$ and choose $x_r \in (\sigma_{\max}^2(1+\sqrt{c})^2, \infty)$. 
        If $\inf \mathrm{supp}(\mu) = 0$, let $x_l$ be any negative number; otherwise, choose $x_l \in (0, \inf \mathrm{supp}(\mu))$.
		  Let $\mathcal{C}_u = \{x + \mathbf{i} v_0: x \in [x_l, x_r]\}$. 
		Then 
		\[
		\mathcal{C} = 
		\{x_l + \mathbf{i} v: v \in [0, v_0]\}
		\cup \mathcal{C}_u
		\cup
		\{x_r + \mathbf{i} v: v \in [0, v_0]\}.
		\]
		To avoid singularities near the real axis, we truncate the contour at a height $n^{-1}\varepsilon_n$, where $\varepsilon_n \downarrow 0$ satisfies $\varepsilon_n \ge n^{-\alpha}$ for some $\alpha \in (0,1]$. 
		Set 
		\[
		\mathcal{C}_l = \{x_l + \mathbf{i} v: v \in [n^{-1}\varepsilon_n, v_0]\}, 
		\qquad 
		\mathcal{C}_r = \{x_r + \mathbf{i} v: v \in [n^{-1}\varepsilon_n, v_0]\},
		\]
		and define $\mathcal{C}_n = \mathcal{C}_l \cup \mathcal{C}_u \cup \mathcal{C}_r$. 
		For $z = x + \mathbf{i} v$, define the truncated process
		\[
		\widehat{M}_n(z) =
		\begin{cases}
			M_n(z), & \text{if } z \in \mathcal{C}_n,\\[3pt]
			M_n(x_l + \mathbf{i}n^{-1}\varepsilon_n), & \text{if } x = x_l,\ v \in [0, n^{-1}\varepsilon_n],\\[3pt]
			M_n(x_r + \mathbf{i} n^{-1}\varepsilon_n), & \text{if } x = x_r,\ v \in [0, n^{-1}\varepsilon_n].
		\end{cases}
		\]
		It is observed that on the subset $\mathcal{C}_n$ of $\mathcal{C}, M_n(\cdot)$ agrees with $\widehat{M}_n(\cdot)$. 
		We now state a CLT for the truncated process $\widehat{M}_n(\cdot)$, which constitutes the key step in proving Theorem~\ref{th1}.
		\begin{lemma}\label{clt_M}
			Under the conditions of Theorem \ref{th1}, $\widehat{M}_n(z)$ forms a tight sequence and converges weakly to a two-dimensional Gaussian process $M(\cdot)$ satisfying for $z \in \mathcal{C} \cup \overline{\mathcal{C}}$ with $\overline{\mathcal{C}}=\{\bar{z}: z \in \mathcal{C}\}$, the mean can be calculated as
			\begin{align}\label{mean_M}
				\mu_n(M(z_1))=\mathscr E_n(z),
			\end{align}
			and for $z_1, z_2 \in \mathcal{C} \cup \overline{\mathcal{C}}$, the covariance function is given by
			\begin{align}\label{cov_M}
				\nu_n\left(M\left(z_1\right), M\left(z_2\right)\right)=\mathscr V_n(z_1,z_2).
			\end{align}
		\end{lemma}
		
	Establishing the preceding lemma directly yields the desired CLT for the LSS. 
    Using standard arguments based on Cauchy’s integral representation, one can verify that, with probability one and uniformly for $f \in \{f_1,\ldots,f_k\}$,
	\[
	\left|\oint f(z)\,\big(M_n(z)-\widehat{M}_n(z)\big)\,\mathrm{d}z\right| \rightarrow 0,
	\]
	where the contour integral is taken over $z \in \mathcal{C} \cup \overline{\mathcal{C}}$.
	
	\medskip
	To prove the lemma, we decompose $M_n(z)$ for $z \in \mathcal{C}_n$ as
	\[
	M_n(z)
	= \sqrt{p}\,q\big[m_n(z) - \mathbb{E}m_n(z)\big]
	+ \sqrt{p}\,q\big[\mathbb{E}m_n(z) - m_n^0(z)\big]
	\;\triangleq\;
	M_{n1}(z) + M_{n2}(z).
	\]
	The analysis of the limiting behavior of $M_n(z)$ therefore reduces to studying the two processes $M_{n1}(z)$ and $M_{n2}(z)$.
	The remainder of the proof proceeds in three steps.

\begin{enumerate}
    \item \textbf{Convergence of the stochastic term.}
    We first show that, for every $z \in \mathcal{C}_n$, the process $M_{n1}(z)$ converges in distribution to a centered Gaussian process with covariance function given by \eqref{cov_M} in Lemma \ref{clt_M}. 
    The argument is based on a martingale difference decomposition together with a martingale CLT.
    \item \textbf{Tightness.}
    We then verify the tightness of the collection $\{M_{n1}(z): z \in \mathcal{C}_n\}$, which guarantees functional convergence of the process.
    \item \textbf{Convergence of the deterministic term.}
    Finally, we show that the deterministic term $M_{n2}(z)$ converges uniformly to its deterministic limit specified in~\eqref{mean_M}.
\end{enumerate}
Combining these three steps yields the central limit theorem for $M_n(z)$, and therefore for the LSS. Further technical details are provided in the supplementary material.

\end{appendix}

\begin{acks}[Acknowledgments]
The authors would like to express their sincere gratitude to Zeqin Lin and Shizhe Hong for their valuable assistance and insightful discussions during the development of several proofs.
\end{acks}

\begin{funding}
Dandan Jiang was  supported by NSFC under Grant No. 12571311.
\end{funding}

\begin{supplement}
\stitle{Supplementary Material for ``Phase Transition of Spectral Fluctuations in Large Gram Matrices with a Variance Profile: A Unified Framework for Sparse CLTs''}
\sdescription{This supplementary material collects the proofs of Lemmas~2.5--2.7, 
along with the detailed proof of Theorem~3.3 and the verification of Remark~3.4.}
\end{supplement}


\bibliographystyle{imsart-number} 
\bibliography{myref}       


\end{document}